\documentclass[12pt,a4paper]{article}
\usepackage{amsmath,amsthm,amsfonts,amssymb,bbm}
\usepackage{graphicx,psfrag,subfigure,color}
\usepackage{cite}
\usepackage{hyperref}

\numberwithin{equation}{section}
\newcommand{\GUE}{\mathrm{GUE}}
\newcommand{\GOE}{\mathrm{GOE}}
\newcommand{\Or}{\mathcal{O}}

\newcommand{\Pb}{\mathbb{P}}
\newcommand{\E}{\mathbbm{E}}

\newcommand{\e}{\varepsilon}

\newcommand{\R}{\mathbb{R}}
\newcommand{\N}{\mathbb{N}}
\newcommand{\Z}{\mathbb{Z}}

\DeclareMathOperator\argmax{argmax}

\newtheorem{prop}{Proposition}[section]
\newtheorem{thm}[prop]{Theorem}
\newtheorem{lem}[prop]{Lemma}

\newtheorem{cor}[prop]{Corollary}

\newtheorem{cla}[prop]{Claim}

\newtheorem{rem}[prop]{Remark}

\title{Limit law of a second class particle in TASEP with non-random initial condition}
\author{P.L. Ferrari\thanks{Institute for Applied Mathematics, Bonn University, Endenicher Allee 60, 53115 Bonn, Germany. E-mail: {\tt ferrari@uni-bonn.de}}
\and P. Ghosal\thanks{Columbia University, Department of Statistics, 1255 Amsterdam Avenue, New York, NY 10027, USA. E-mail: {\tt pg2475@columbia.edu}}
\and P. Nejjar\thanks{IST Austria, Am Campus 1, 3400 Klosterneuburg, Austria.
 Partially supported by ERC Advanced Grant No. 338804 and ERC Starting Grant No. 716117, E-mail: {\tt peter.nejjar@ist.ac.at}}}
\date{May 22, 2018}

\begin{document}
\sloppy
\maketitle

\begin{abstract}
We consider the totally asymmetric simple exclusion process (TASEP) with non-random initial condition and  density $\lambda$ on $\Z_-$ and $\rho$ on $\Z_+$,
and a second class particle initially at the origin. For $\lambda<\rho$, there is a shock and the second class particle moves with speed $1-\lambda-\rho$.
For large time $t$, we show that the position of the second class particle fluctuates on the  $t^{1/3}$ scale and determine its limiting law. We also obtain
the limiting distribution of the number of steps made by the second class particle until time $t$.
\end{abstract}



\section{Introduction and main result}\label{sectIntro}
The totally asymmetric simple exclusion process (TASEP) is one of  the simplest non-reversible interacting particle systems on $\mathbb{Z}$ lattice. A TASEP particle configuration is described by the occupation variables $\{\eta_j\}_{j\in \mathbb{Z}}$ where $\eta_j=1$ means site $j$ is occupied by a particle and $\eta_j=0$ means site $j$ is empty. The TASEP dynamics  are  as follows.  Particles jump one step to the right and are allowed to do so only if their right neighboring site is empty. Jumps are independent of each other and are performed after an exponential waiting time with mean $1$, which starts from the time instant when the right neighbor site is empty.

More precisely, denoting by $\eta\in {\cal S}=\{0,1\}^\Z$ a configuration, the infinitesimal generator of TASEP is given by the closure of the operator $L$
\begin{equation}
L f(\eta)=\sum_{j\in\Z}\eta_j (1-\eta_{j+1})(f(\eta^{j,j+1})-f(\eta)),
\end{equation}
where $f:{\cal S}\to\R$ is any function depending on finitely many $\eta_j$'s, and $\eta^{j,j+1}$ is the configuration $\eta$ with the occupation variables at sites $j$ and $j+1$ interchanged. The semigroup $e^{L t}$ is well-defined as acting on bounded and continuous functions on ${\cal S}$, see~\cite{Li85b} for further details on the construction. Since particles can not overtake each other, we can associate a labeling to them and denote the position of particle $k$ at time $t$ by $x_k(t)$. We choose the right-to-left ordering, namely, we have $x_{k+1}(0)< x_k(0)$ for any $k$.

The observable we study in this paper is the so-called \emph{second class particle}. The only difference between a particle described above (also called first class particle) with a second class particle is the following: when a first class particle tries to jump on a site occupied by a second class particle, the jump is not suppressed and the two particles interchanges their positions. Second class particles can be also seen as discrepancies between two TASEP systems coupled by the basic coupling, see e.g.~\cite{Li99}, starting from p.218, for further  details. For instance, consider two TASEP configurations $\eta, \eta^{\prime}\in {\cal S}$ which at time $0$ differ only  at the site $0$. Then under basic coupling the two configurations will differ only at one site for any time $t\geq 0$, and we can define
\begin{equation}\label{def2ndclass}
X_t= \sum_{x \in \Z}x \mathbbm{1}_{\{\eta_t (x) \neq \eta^{\prime}_t (x)\}}.
\end{equation}
to be the position of the second class particle at time $t$.
 From this perspective, the distribution of $X_t$ gives an information on the correlation between $\eta_j(t)$ and $\eta_0(0)$. For the case of the stationary initial condition, where every site is independently occupied with probability $\varrho$, the law of the second class particle is precisely proportional to ${\rm Cov}(\eta_j(t),\eta_0(0))$, see e.g.~\cite{PS01}.

Second class particle are also very useful when the interacting system generates shocks, which are discontinuities in the particle density. Indeed, the position of the shock can be identified with $X_t$, see Chapter~3 of~\cite{Li99}. For that reason already several studies are devoted to the second class particles. For instance, consider a single second class particle starting at the origin. For TASEP initial condition is Bernoulli product measure with density $\lambda$ on $\Z_-=\{\ldots,-3,-2,-1\}$ and density $\rho$ on $\Z_+=\{1,2,3,\ldots\}$ with $\lambda<\rho$. Then $X_t$ moves with an average speed of $1-\rho-\lambda$ and it has Gaussian fluctuations on a $t^{1/2}$ scale~\cite{Fer90,FF94b,PG90,FMP09}.  For shocks created  by deterministic initial data,~\cite{FN13}  obtained the limit law of a first class particle located at the shock. Later, multipoint distributions of several particles at the shock where obtained~\cite{FN16}, and the transition to shock fluctuations was studied~\cite{N17}. For the case $\lambda<\rho$, what remained open is the limit law of the second class particle $X_t$, and this is what we obtain in this paper.

When the initial condition is product Bernoulli measure with $\lambda$ on $\mathbb{Z}_{-}$ and $\rho$ on $\mathbb{Z}_{+}$ with $\lambda> \rho$, a second class particle placed
initially at the origin chooses one direction uniformly among its characteristics and then moves at a constant speed along that direction~\cite{FK95,MG05}. This means that the
fluctuations of the second class particle are of order $t$. Tightly related with the trajectory of the second class particle is the notion of competition interface in a last
passage percolation model introduced in~\cite{FP05B} for $\lambda=1$ and $\rho=0$. This connection was later extended for more general cases in~\cite{FMP09}. They prove that
for any initial conditions with asymptotically density $\lambda<\rho$ to the left and right of the origin respectively, a second class particle moves along a characteristics,
chosen according to a random distribution which depends on the initial condition. However only for the Bernoulli initial condition is the precise distribution known.

For constant density results on the fluctuations of the second class particle are known only for the stationary case. Using coupling arguments it was shown in~\cite{BS08} that
the fluctuation scale of the second class particle is $t^{2/3}$ (by controlling some moments). The exact limiting distribution has been obtained in~\cite{FS05a}.

\subsubsection*{Main result}
In this paper we also consider TASEP with initial density $\lambda$ on $\Z_-$ and $\rho$ on $\Z_+$, $\lambda<\rho$, but this time \emph{without initial randomness}. More precisely, we start with TASEP particles at
\begin{equation}\label{eq:InitialPosition}
 \begin{aligned}
 x_{n}(0)&=-\left\lfloor \lambda^{-1} n\right\rfloor&\quad \text{when }n>0,\\
 x_{n}(0)&=-\left\lfloor \rho^{-1} n \right\rfloor& \quad \text{when }n<0,
 \end{aligned}
\end{equation}
and one second class particle at the origin, $X_0=0$. Our main result, Theorem~\ref{MainTheoremForDifferentDensity}, concerns the joint limiting distribution of $X_t$ and of $N_t$, where the latter being the number of steps made by the second class particle until time $t$.
\begin{thm}\label{MainTheoremForDifferentDensity}
Consider TASEP with initial condition (\ref{eq:InitialPosition}), where  $0<\lambda<\rho<1, $ and  a  second class particle  at $0$ at time $t=0$. Let  $X_t$ be the second class particle's position at time $t$, and let $N_t$ be the number of steps it made up to time $t$.
Then, we have convergence in distribution
\begin{equation}\label{eq:ConvJointLaw}
\left(\frac{X_t- \mathbf{v}t }{t^{1/3}},\frac{N_t- 2\mu_0^{-1} t}{t^{1/3}}\right)\Rightarrow (\mathcal{X}, \mathcal{N})
 \end{equation}
as $t\to\infty$, where
 \begin{equation}\label{eq:SecondClassDistribution}
 \mathcal{X}=\frac{2^{1/3}}{\mu^{4/3}_0\Upsilon} \left(\frac{\sigma_1\xi_{\rm GOE}^{(1)}}{\rho(1-\rho)}-\frac{\sigma_2\xi_{\rm GOE}^{(2)}}{\lambda(1-\lambda)}\right),
 \end{equation}
 and
 \begin{equation}\label{eq:NumberoOfJumpsDistribution}
 \mathcal{N}= \frac{2^{1/3}}{\mu^{4/3}_0\Upsilon} \left(\frac{\sigma_1(1-2\rho)\xi_{\rm GOE}^{(1)}}{\rho(1-\rho)}-\frac{\sigma_2(1-2\lambda)\xi_{\rm GOE}^{(2)}}{\lambda(1-\lambda)}\right).
 \end{equation}
Here, $\xi_{\rm GOE}^{(1)}$ and $\xi_{\rm GOE}^{(2)}$ are two \emph{independent} GOE Tracy-Widom distributed random variables~\cite{TW96}, and the constants are given by
 \begin{equation}\label{const1}
 \mathbf{v}:=1-\lambda-\rho,\quad\mu_0:= \frac{2}{1-\lambda-\rho+2\lambda\rho}, \quad\Upsilon:= \frac{1-2\lambda}{\lambda(1-\lambda)} - \frac{1-2\rho}{\rho(1-\rho)},
 \end{equation}
 and
 \begin{equation}\label{const2}\sigma_1:= \frac{2^{1/3}}{\left(\lambda(1-\lambda)(1-\lambda-\rho+2\lambda\rho)\right)^{1/3}},
 \quad \sigma_2:= \frac{2^{1/3}}{\left(\rho(1-\rho)(1-\lambda-\rho+2\lambda\rho)\right)^{1/3}}.
\end{equation}
\end{thm}

One aspect of Theorem~\ref{MainTheoremForDifferentDensity} is the form of the limiting distribution. It is a linear combination of two independent random variables. This is directly related to the fact that  the shock is the location where two characteristic lines meet. Secondly, the $t^{1/3}$ scale is indirectly related with the fact that TASEP is in the Kardar-Parisi-Zhang (KPZ) universality class~\cite{KPZ86} in $1+1$ dimensions (see surveys and lecture notes~\cite{FS10,Cor11,QS15,BG12,Qua11,Fer10b}). Finally the GOE Tracy-Widom distribution is due to fact that the regions from where the two characteristic lines starts are have fixed densities and they are non-random (such situations are called flat initial conditions for KPZ models~\cite{PS00,Fer04}, see Proposition~\ref{ComponentConvergenceProp}).

Further, Theorem~\ref{MainTheoremForDifferentDensity} indeed shows that that the $t^{1/2}$ fluctuation of $X_t$ (and the Gaussian distribution) observed with Bernoulli-Bernoulli initial condition strongly depends on the randomness of the initial condition. The reason is that the fluctuations generated by the dynamics (of KPZ type) are not visible as they are much smaller than the Gaussian fluctuations of the initial condition of the two regions which are strongly correlated with the shock. Finally, notice that when $\lambda-\rho\to 0$, the constant  $\Upsilon$ in \eqref{const1} converges to $0$. This is not surprising as a similar phenomenon has been observed for the random initial condition, $\lambda-\rho\to 0$ corresponds to look at the stationary initial condition, where the position of the second class particle has a non-trivial limiting distribution in the $t^{2/3}$ scale~\cite{FS05a,PS01}.

On the way to  Theorem~\ref{MainTheoremForDifferentDensity}, we determine a technical result which, however, has its own interest. In Lemma~\ref{AlternativeExitPointLemma} we find a Gaussian bound for the tail of the distribution of the starting point (or, exit point) of a geodesic in the stationary model of last passage
percolation (defined in \eqref{eq:LastPassageTimeFromLHalf}) from a line to a point.

\subsubsection*{Method and outline}
The result of the paper uses as first ingredient the asymptotic independence of the last passage times from two disjoint initial set of points of a last passage percolation (LPP) model~\cite{FN13}, see Proposition~\ref{ComponentConvergenceProp}, which is a consequence of slow-decorrelation~\cite{Fer08,CFP10b}. The trajectory of the second class particle is the same as the competition interface in LPP introduced by Ferrari and Pimentel in~\cite{FP05B}. It is however not straightforward to extend the result on the competition interface of Ferrari and Nejjar, Theorem~2.1 of~\cite{FN16}, to the position of the second class particle. The reason being that between the two observables there is a random time change and the second class particle moves with non-zero speed.

The next key ingredient in the proof is a tightness-type result on the two LPP problems, see Proposition~\ref{Tightness} and Corollary~\ref{UsefulCorBis}. This extends to general densities the method of Pimentel~\cite{Pi17}, which extends the comparison with the stationary process by Cator and Pimentel~\cite{CP15b}, see Lemma~\ref{CouplingLemma}. For its application one has to estimate the tail probabilities of the exit points, see Lemma~\ref{AlternativeExitPointLemma} and Lemma~\ref{ExitPointOfGeneralPercolation}. The latter as well as the limiting GOE Tracy-Widom distributions are taken from the recent universality result for flat initial condition by Ferrari and Occelli~\cite{FO17v1}.

The rest of the paper is organized as follows. In Section~\ref{sectTASEPandLPP} we recall the connection between TASEP and LPP as well between second class particle and competition interface. We then collect asymptotic result on LPP which are used later. Section~\ref{SectProofMainResult} contains the proof of Theorem~\ref{MainTheoremForDifferentDensity}. It is based of preliminary results on the control of LPP at different points, whose proof is presented after the one of the main theorem.

\subsubsection*{Acknowledgments}
The authors are grateful to Leandro Pimentel for discussions on his preprint~\cite{Pi17}. P.L.~Ferrari is supported by the German Research Foundation via the SFB 1060--B04 project. P. Ghosal was partially funded by NSF DMS:1637087 during his visit to the workshop "Quantum Integrable systems, Conformal Field Theories and Stochastic Processes"  held at the \emph{Institut d'Etudes Scientifiques de Carg\`ese} where this work has been initiated. P. Ghosal was also partially funded by the grant NSF DMS:1664650 during the preparation of this paper. P. Ghosal also wishes to thank Ivan Corwin, Jeremy Quastel, Mustazee Rahman for helpful conversations. P. Nejjar equally thanks the organizers of the aforementioned workshop in Carg\`ese.

\section{Some results on TASEP and LPP}\label{sectTASEPandLPP}
Here we first recall the connection between TASEP and LPP, then between the second class particle and the competition interface, and finally we collect some known result about LPP.

\subsection{Last passage percolation and connection with TASEP}\label{Sec2Sub1}
We consider last passage percolation (LPP) on $\Z^2$. To define it, let $\mathcal{L}_S ,\mathcal{L}_E\subset \Z^2$. An up-right path $\pi$ from $\mathcal{L}_S$ to $\mathcal{L}_E$ is a sequence of points $\pi=(\pi(0),\ldots,\pi(n))$ in $\Z^{2}$ such that
 $\pi(0)\in \mathcal{L}_{S},\pi(n)\in \mathcal{L}_{E}$ and $\pi(i)-\pi(i-1)\in \{(0,1),(1,0)\},i=1,\ldots,n.$ Consider a family of independent non-negative random variables $\{\omega_{i,j}\}_{i,j\in \Z}.$
 Then the last passage percolation time from $\mathcal{L}_S$ to $\mathcal{L}_E$ is defined as
\begin{align}\label{eq:LPPdef}
L_{ \mathcal{L}_S \to \mathcal{L}_E}=\max_{\pi: \mathcal{L}_S\to \mathcal{L}_E}\sum_{(i,j)\in \pi\setminus \mathcal{L}_S}\omega_{i,j},
\end{align}
where the maximum is taken over all up-right paths from $\mathcal{L}_S$ to $\mathcal{L}_E$. In case there are infinitely many or zero such paths, we set, say, $L_{ \mathcal{L}_S \to \mathcal{L}_E}=-\infty$. We denote by $\pi^{\mathrm{max}}$ a path for which the maximum in \eqref{eq:LPPdef} is attained; as we will consider continuous $\omega_{i,j}$, $\pi^{\mathrm{max}}$ will be a.s. unique.

Let $\{x_{n}(0),n \in I\}, I \subset \Z$ be a TASEP initial data and consider the line
\begin{equation} \label{ell}
{\cal L}:= \{(k+x_k(0),k): k\in I\}.
\end{equation}
Let furthermore $\{\omega_{i,j}\}_{(i,j) \in \Z\times I}$ be independent and distributed as
\begin{align}\label{eq:RandomWeights}
\omega_{i,j}\sim \mbox{Exp}(v_j)
\end{align}
where $v_j$ is the parameter for the exponential waiting time of particle number $j.$
Then, for any integer $\ell$, the link between TASEP and the last passage percolation is given by
\begin{align}\label{eq:LPP&TASEPcoupling}
\Pb\left(\bigcap_{k=1}^\ell \{L_{{\cal L}\to (m_k,n_k)}\leq t\}\right)=\Pb\left(\bigcap_{k=1}^\ell \{x_{n_k}(t)\geq m_k-n_k\}\right).
\end{align}

\subsection{Competition interface and second class particles}
Given an initial data $\{x_{n}(0),n \in \Z\}$ we consider separately the sets
\begin{equation}\label{eq:HalfLines}
\mathcal{L}^{+}=\{(k+x_{k}(0),k):k>0\} \quad \mathcal{L}^{-}=\{(k+x_{k}(0),k):k\leq 0\}.
\end{equation}
We assume that
$x_{0}(0)=0$ and $x_{1}(0)<-1$, which guarantees that $L_{\mathcal{L}^{+}\to (m,n)}\neq L_{\mathcal{L}^{-}\to (m,n)}$ for all $(m,n) \in \Z_{\geq 1}^{2}.$ We consider the two clusters
\begin{equation}
\begin{aligned}
&\Gamma_{+}^{\infty}:=\{(i,j)\in \Z^{2}:L_{\mathcal{L}^{+}\to (i,j)}>L_{\mathcal{L}^{-}\to (i,j)}\},\\
& \Gamma_{-}^{\infty}:=\{(i,j)\in \Z^{2}:L_{\mathcal{L}^{-}\to (i,j)}>L_{\mathcal{L}^{+}\to (i,j)}\}.
\end{aligned}
\end{equation}
Note that by assumption for each $(m,n)\in \Z_{\geq 1}^{2}$ we have $(m,n) \in \Gamma_{+}^{\infty}\cup \Gamma_{-}^{\infty}$ almost surely.
These two clusters are separated through the competition interface $\phi=(\phi_0,\phi_1,\phi_2,\ldots),$ defined by setting $\phi_0 =(0,0)$ and
\begin{equation}\label{compint}
 \phi_{n+1}=\begin{cases}
 \phi_{n}+(1,0) \quad \mathrm{ if }\quad \phi_{n}+(1,1) \in \Gamma_{+}^{\infty},\\
 \phi_{n}+(0,1) \quad \mathrm{ if }\quad \phi_{n}+(1,1) \in \Gamma_{-}^{\infty},
 \end{cases}
\end{equation}
We write $\phi_{n}=(I_n,J_n),n \geq 0,$ by definition $I_n + J_n=n$. In our construction, we always have $\phi_1 =(1,0)$.
Now, in~\cite{FP05B} a coupling was introduced between the competition interface and the second class particle, see~\cite{FMP09}, Proposition 2.2, for the statement for general TASEP initial data. In the construction of~\cite{FMP09}, TASEP with a second class particle initially at the origin is seen to be equivalent to TASEP without a second class particle, but an extra site, and initially a hole at the origin and a particle at site 1. Associating to this initial data the sets \eqref{ell}, \eqref{eq:HalfLines} and the competition interface, $I_n - J_n -1$ becomes the position of the second class particle after its $(n-1)\mathrm{th}$ step, i.e. by setting $\tau_{n}=L_{\mathcal{L} \to \phi_n}$ and defining
\begin{equation}
(I(t),J(t))=\phi_n ,\quad t \in [\tau_n, \tau_{n+1})
\end{equation}
the coupling of~\cite{FMP09} makes that
\begin{equation}\label{eq:Corr2ndComp}
(I(t)-J(t)-1)_{t \geq 0} = (X_t)_{t \geq 0}
\end{equation}
 (the shift by $-1$ is not present in~\cite{FMP09}, due to slightly different conventions).

\subsection{Some asymptotic results for LPP}\label{Sec2Sub2}
We consider i.i.d.\ weights $\{\omega_{i,j}\}_{i,j\in \Z}$ with $\omega_{i,j } \sim \mbox{Exp(1)}$, where
 $\mbox{Exp}(1)$ denotes the exponential distribution with intensity $1$.
The basic result about convergence of the point-to-point LPP time is the following.
\begin{prop}[Theorem~1.6 of~\cite{Jo00b}]\label{propPtToPt}
Let $0<\eta<\infty$. Then,
\begin{align}\label{eq:LPPPointToPoint}
\lim_{N\to \infty} \Pb\left(L_{0\to (\lfloor \eta N\rfloor, \lfloor N\rfloor}\leq \mu_{pp}N+s \sigma_{\eta} N^{1/3}\right) = F_\GUE (s)
\end{align}
where $\mu_{pp} = (1+\sqrt{\eta})^2$, and $\sigma_{\eta} = \eta^{-1/6}(1+\sqrt{\eta})^{4/3}$, and $F_\GUE$ is the GUE Tracy-Widom distribution function.
 \end{prop}
Further, we need the following upper/lower tail estimates for the distribution of $L_{0\to (\eta N, N)}$. These estimates were collected in~\cite{FN13}, based on results from~\cite{BBP06,BFP12,FS05a}.

\begin{prop}[Proposition 4.2 in~\cite{FN13}]\label{ppUpTail} Let $0<\eta<\infty$. Then for any given $n_0>0$ and $s_0\in \R$, there exists constants $C,c>0$ only dependent on $N_0,s_0$ such that for all $N\geq N_0$ and $s\geq s_0$, we have
 \begin{align}\label{eq:ppUpTailBd}
 \Pb\left(L_{0\to (\lfloor \eta N\rfloor, \lfloor N\rfloor)}>\mu_{pp}N +s N^{1/3}\right)\leq C\exp(-cs)
\end{align}
 where $\mu_{pp}=(1+\sqrt{\eta})^2$.
 \end{prop}

 \begin{prop}[Proposition 4.3 in~\cite{FN13}]\label{ppLowTail} Let $0<\eta<\infty$ and $\mu_{pp}=(1+\sqrt{\eta})^2$. There exists positive constants $s_0,N_0,C,c$ such that the following holds
 \begin{align}
 \Pb\left(L_{0\to (\lfloor \eta N\rfloor, \lfloor N\rfloor )}\leq \mu_{pp}N+sN^{1/3}\right)\leq C\exp(-c|s|^{3/2})
 \end{align}
 for all $s\leq -s_0, N\geq N_0$.
 \end{prop}

Recall our initial data \begin{equation}
 \begin{aligned}
 x_{n}(0)&=-\left\lfloor \lambda^{-1} n\right\rfloor& \quad \text{when }n>0,\\
 x_{n}(0)&=-\left\lfloor \rho^{-1} n \right\rfloor& \quad \text{when } n\leq 0,
 \end{aligned}
 \end{equation}
 where $0<\lambda<\rho<1$. This creates a shock into the system.
We consider the following anti-diagonal lines
\begin{align}
{\cal L}_{\lambda}:= \left\{\left(\big\lfloor\tfrac{\lambda-1}{\lambda} x\big\rfloor, x\right): x\in \Z_{\geq 0}\right\},
\quad {\cal L}_{\rho}:= \left\{\left(\big\lfloor \tfrac{\rho-1}{\rho} x\big\rfloor, x\right) : x\in \Z_{<0}\right\}.
\end{align}

The next result consists in the limiting law from ${\cal L}_\rho$ and ${\cal L}_\lambda$.
\begin{lem}\label{Lemma1} Let $0< \lambda < \rho < 1$ be fixed. Define the constants
\begin{equation}
\begin{aligned}\label{parameter}
\gamma:= \frac{1-\lambda-\rho}{1-\lambda-\rho+2 \lambda\rho}, &\quad \mu_0:= \frac{2}{1-\lambda-\rho+2\lambda\rho},\\ \sigma_1:= \frac{2^{1/3}}{\left(\lambda(1-\lambda)(1-\lambda-\rho+2\lambda\rho)\right)^{1/3}}, & \quad \sigma_2:= \frac{2^{1/3}}{\left(\rho(1-\rho)(1-\lambda-\rho+2\lambda\rho)\right)^{1/3}}.
\end{aligned}
\end{equation}
Then, we have
\begin{align}\label{eq:Lemma1}
 \lim_{N\to \infty}\Pb\left(\frac{L_{{\cal L}_{\lambda}\to((1+\gamma) N, (1-\gamma)N )}-\mu_0 N}{N^{1/3}}\leq s\right)=F_{\GOE}\left(\frac{2^{2/3} s}{\sigma_1}\right),
\end{align}
and
\begin{align}\label{eq:Lemma2}
 \lim_{N\to \infty}\Pb\left(\frac{L_{{\cal L}_{\rho}\to((1+\gamma) N, (1-\gamma)N )}-\mu_0 N}{N^{1/3}}\leq s\right)=F_{\GOE}\left(\frac{2^{2/3} s}{\sigma_2}\right).
\end{align}
\end{lem}

\begin{proof}
Let us define ${\cal L}_\rho^\infty=\{(\lfloor(\rho-1) x/\rho\rfloor,x),x\in\Z\}$ the two-sided extension of ${\cal L}_\rho$. Let $Q$ the projection of
\begin{equation}\label{eqP}
P=((1+\gamma)N,(1-\gamma)N)
\end{equation}
along the $(1,1)$ direction onto ${\cal L}_\rho^\infty$, as illustrated in Figure~\ref{FigShift}.
\begin{figure}
\begin{center}
\psfrag{L}[cl]{${\cal L}_\rho$}
\psfrag{Lt}[cl]{$\widetilde {\cal L}_\rho$}
\psfrag{P}[cl]{$P$}
\psfrag{Pt}[cl]{$\widetilde P$}
\psfrag{At}[cl]{$\widetilde A$}
\psfrag{Q}[cl]{$Q$}
\includegraphics[height=6cm]{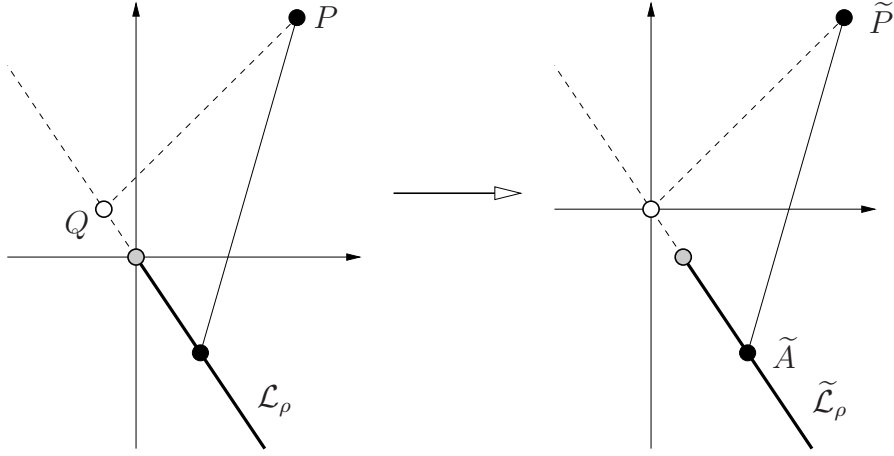}
\caption{Illustration of the mapping used in the proof of Lemma~\ref{Lemma1}. }
\label{FigShift}
\end{center}
\end{figure}
It is given by
\begin{equation}
Q=(2(1-\rho)\gamma N,-2\gamma\rho N).
\end{equation}
Now we set $\tilde N=(1-\gamma(1-2\rho)) N$ and we considering $Q$ as the new origin. Then ${\cal L}_\rho$ is mapped into $\widetilde {\cal L}_\rho=\{(\lfloor(\rho-1) x/\rho\rfloor,x),x\in\Z\cap\{x\leq 2\gamma\rho N\}\}$ and the point $P$ is mapped into $\widetilde P=(\tilde N,\tilde N)$. Thus
\begin{equation}\label{eq4.13}
\lim_{N\to\infty}\Pb\left(\frac{L_{{\cal L}_{\rho}\to((1+\gamma) N,(1-\gamma)N)}-\mu_0 N}{N^{1/3}}\leq s\right) =
\lim_{\tilde N\to\infty}\Pb\left(\frac{L_{\widetilde {\cal L}_{\rho}\to(\tilde N,\tilde N)}-a_0 \tilde N}{a_1 \tilde N^{1/3}}\leq \frac{s}{\sigma_1}\right)
\end{equation}
with $a_0=1/(\rho(1-\rho))$ and $a_1=1/(\rho(1-\rho))^{2/3}$. In Theorem~2.1 of~\cite{FO17v1} it is proven that
\begin{equation}\label{eq4.14}
\lim_{\tilde N\to\infty}\Pb\left(\frac{L_{{\cal L}_{\rho}^\infty\to(\tilde N,\tilde N)}-a_0 \tilde N}{a_1 \tilde N^{1/3}}\leq s\right) = F_{\GOE}(2^{2/3} s).
\end{equation}
To obtain \eqref{eq4.14}, one need to show that with high probability the maximizer is localized in a $M \tilde N^{2/3}$-neighborhood of $\widetilde A=(\tilde N(2\rho-1)/\rho,-\tilde N(2\rho-1)/(1-\rho))$. This was shown in Lemma~4.3 of~\cite{FO17v1}. Due to the condition $\rho>\lambda$, the starting point of the line $\widetilde {\cal L}_\rho$ is of order $\tilde N$ to the left of $\widetilde A$. Thus one has
\begin{equation}\label{eq4.15}
\lim_{\tilde N\to\infty}\Pb\left(\frac{L_{\widetilde {\cal L}_{\rho}\to(\tilde N,\tilde N)}-a_0 \tilde N}{a_1 \tilde N^{1/3}}\leq \frac{s}{\sigma_1}\right) =
\lim_{\tilde N\to\infty}\Pb\left(\frac{L_{{\cal L}_{\rho}^\infty\to(\tilde N,\tilde N)}-a_0 \tilde N}{a_1 \tilde N^{1/3}}\leq \frac{s}{\sigma_1}\right).
\end{equation}
Combining (\ref{eq4.13})-(\ref{eq4.15}) we get (\ref{eq:Lemma1}). The statement (\ref{eq:Lemma2}) is proven analogously.
\end{proof}

By maximizing the point-to-point law of large numbers, one gets that the maximizers from ${\cal L}_\rho$ to $P$ (defined in (\ref{eqP})) starts, on the $\Or(N)$ scale, from
\begin{equation}\label{eqA}
A_\rho=((\rho-1)/\rho,1)\xi_\rho N,\quad \xi_\rho=\frac{-2(\rho-\lambda)\rho}{1-\lambda-\rho+2\lambda\rho}.
\end{equation}
The straight line passing from $A_\rho$ to $P$ (defined in \eqref{eqP}) is the characteristic and it has direction given by $((1-\rho)^2,\rho^2)$. Similarly, the maximizers from ${\cal L}_\lambda$ to $P$ starts, on the $\Or(N)$ scale, from
\begin{equation}\label{eqB}
A_\lambda=((\lambda-1)/\lambda,1)\xi_\lambda N,\quad \xi_\lambda=\frac{2(\rho-\lambda)\lambda}{1-\lambda-\rho+2\lambda\rho}.
\end{equation}
Slow decorrelation~\cite{Fer08,CFP10b} says that the fluctuations from ${\cal L}_\rho$ and ${\cal L}_\lambda$ to the end-point are almost the same as the ones from these lines to a point of the characteristic line at a $o(N)$ distance to the end-point. Therefore, for any $0<\nu<1$, we define the points on the characteristics
\begin{eqnarray}
E_\rho &=&((1+\gamma) N-(1-\rho)^2 N^{\nu}, (1-\gamma)N-\rho^2 N^\nu),\\
E_\lambda &= &((1+\gamma) N- (1-\lambda)^2 N^{\nu}, (1-\gamma)N-\lambda^2 N^\nu),
\end{eqnarray}
see Figure~\ref{FigDecorr}.
\begin{figure}
\begin{center}
\psfrag{Lr}[cl]{${\cal L}_\rho$}
\psfrag{Ll}[cl]{${\cal L}_\lambda$}
\psfrag{P}[cl]{$P$}
\psfrag{Al}[cr]{$A_\lambda$}
\psfrag{Ar}[cl]{$A_\rho$}
\psfrag{El}[cr]{$E_\lambda$}
\psfrag{Er}[cl]{$E_\rho$}
\includegraphics[height=6cm]{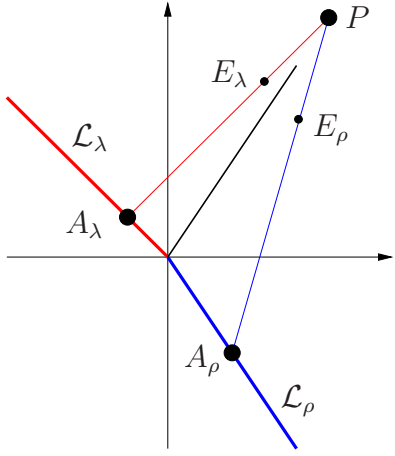}
\caption{Illustration of the geometric setting of Lemma~\ref{Lemma2} and~\ref{Lemma3}. The black solid line represents the set of points $\cup_{\eta\in {\cal J}_{\beta,N}} D_\eta$.}
\label{FigDecorr}
\end{center}
\end{figure}

\begin{lem}\label{Lemma2}
There exists some $\tilde \mu_1(\gamma)$ and $\tilde \mu_2(\gamma)$ such that
\begin{equation}\label{eq:Lemma3}
\begin{aligned}
\lim_{N\to\infty} \Pb\left(\frac{L_{E_\lambda\to((1+\gamma) N, (1-\gamma)N)}-\tilde \mu_1 N^\nu}{N^{\nu/3}}\leq s\right)&=G_0(s),\\
\lim_{N\to\infty} \Pb\left(\frac{L_{E_\rho\to((1+\gamma) N, (1-\gamma)N)}-\tilde \mu_2 N^\nu}{N^{\nu/3}}\leq s\right)&=G^\prime_0(s),
\end{aligned}
\end{equation}
and
\begin{equation}\label{eq:Lemma4}
\begin{aligned}
\lim_{N\to\infty} \Pb\left(\frac{L_{{\cal L}_\lambda\to E_\lambda}-\mu_0 N+\tilde \mu_1 N^\nu}{N^{1/3}}\leq s\right)&= F_{\GOE}\left(\frac{2^{2/3}s}{\sigma_1}\right),\\
\lim_{N\to\infty} \Pb\left(\frac{L_{{\cal L}_\rho\to E_\rho}-\mu_0 N+\tilde \mu_2 N^\nu}{N^{1/3}}\leq s\right)&=F_{\GOE}\left(\frac{2^{2/3}s}{
\sigma_2}\right),
\end{aligned}
\end{equation}
where $G_0$ and $G^\prime_0$ are GUE Tracy-Widom distributions up to a scaling factor.
\end{lem}
\begin{proof}
(\ref{eq:Lemma3}) follows from Proposition~\ref{propPtToPt}, while (\ref{eq:Lemma4}) are obtained as in Lemma~\ref{Lemma1}.
\end{proof}

\begin{lem}\label{Lemma3}
 Fix any $2/3<\beta<1$. For $\eta\in \mathcal{J}_{\beta,N}:=[0,1-N^{\beta-1}]$ define the points $D_\eta=(\lfloor \eta (1+\gamma) N\rfloor , \lfloor \eta(1-\gamma) N\rfloor)$. Then
\begin{equation}\label{eq:Lemma5}
 \begin{aligned}
 \lim_{N\to\infty}\Pb\Bigg(\bigcup_{\eta\in \mathcal{J}_{\beta,N}}\left\{D_\eta\in \pi^{\max}_{L_{{\cal L}_{\lambda}\to E_\lambda}}\right\}\Bigg)=0,\\
 \lim_{N\to\infty}\Pb\Bigg(\bigcup_{\eta\in \mathcal{J}_{\beta,N}}\left\{D_\eta\in \pi^{\max}_{L_{{\cal L}_{\rho}\to E_\rho}}\right\}\Bigg)=0.
 \end{aligned}
\end{equation}
\end{lem}
\begin{proof}
We prove the statement for $\pi^{\max}_{L_{{\cal L}_{\rho}\to E_\rho}}$, the other being similar.
Let us set $\varepsilon=N^{-\chi_1}$ for some $0<\chi_1<1/3$. Recall $\xi_\rho$ from \eqref{eqA}. For the event
 \begin{equation}
 B(N, \varepsilon):= \{ \pi^{\mathrm{max}}_{\mathcal{L}^{\rho}\to E_{\rho}}(0)\in \{ ((\rho-1)x/\rho, x): x \leq (\xi_\varrho+\varepsilon )N \} \}
 \end{equation}
 we have $\lim_{N \to \infty}\Pb(B(N, \varepsilon))=1$ by Lemma~4.3 of~\cite{FO17v1}. Consider the point \mbox{$R=(\lfloor (\rho-1)(\xi_\varrho+\varepsilon )N/\rho\rfloor, \lfloor (\xi_\varrho+\varepsilon )N\rfloor)$}.
 Then
 \begin{equation}
 B(N, \varepsilon)\cap \bigcup_{\eta\in \mathcal{J}_{\beta,N}}\left\{D_\eta\in \pi^{\max}_{L_{{\cal L}_{\rho}\to E_\rho}}\right\}\subseteq \bigcup_{\eta\in \mathcal{J}_{\beta,N}}\left\{D_\eta\in \pi^{\max}_{R\to E_\rho}\right\},
 \end{equation}
since $ \pi^{\max}_{L_{{\cal L}_{\rho}\to E_\rho}}$ can coalesce with, but not cross $ \pi^{\max}_{R\to E_\rho}$.
Define the event $Y(N)$ that $\pi^{\max}_{R\to E_\rho}$ contains a point of the straight line (in $\Z^{2}$) joining $R+(0,N^{\chi_2})$ with $E_\rho+(0,N^{\chi_2})$, for some $1>\chi_2>2/3$. By Theorem~3.5 of~\cite{N17}, we have that $\Pb(Y(N))$ converges to $0$.
Choose now $1-\chi_1<\chi_2<\beta$. The straight line joining $R+(0,N^{\chi_2})$ with $E_\rho+(0,N^{\chi_2})$ crosses the line from $(0,0)$ to $((1+\gamma) N,(1-\gamma)N)$ at euclidean distance $\Or(N^{\chi_2})$ from $ ((1+\gamma) N,(1-\gamma)N)$. Since $\Pb(Y(N))$ converges to $0$, so does $\Pb\left( \bigcup_{\eta\in \mathcal{J}_{\beta,N}}\left\{D_\eta\in \pi^{\max}_{R\to E_\rho}\right\}\right)$, finishing the proof.
\end{proof}

\section{Proof of the main theorem}\label{BodyOfheProof}
In Section~\ref{SectPreliminary} we provide some results on the limiting LPP problems as well as tightness-type results. The proofs are mostly postponed to Section~\ref{AuxilliaryResults}. These results are then employed to prove the main result, Theorem~\ref{MainTheoremForDifferentDensity} in Section~\ref{SectProofMainResult}.

\subsection{Preliminary results}\label{SectPreliminary}
The very first result of this section corresponds to the scaling limit of the last passage times from two half lines $\cal{L}_{\lambda}$ and $\cal{L}_{\rho}$ to the points (on the line $x+y=2N$) lying in a neighborhood of $((1+\gamma)N, (1-\gamma)N)$ (see (\ref{parameter}) for $\gamma$). Before stating the result, we introduce some notations. Recall that $P=((1+\gamma)N,(1-\gamma)N)$. Any point on the line $x+y=2N$ can be written as $((1+\gamma)N+U,(1-\gamma)N-U)$ for some $U\in \R$. For convenience, we denote
\begin{equation}\label{eq:defPU}
P(U):=\left((1+\gamma)N+U, (1-\gamma)N-U\right),
\end{equation}
so that $P=P(0)$, see Figure~\ref{FigGeneralP} for an illustration.
\begin{figure}
\begin{center}
\psfrag{Lr}[cl]{${\cal L}_\rho$}
\psfrag{Ll}[cl]{${\cal L}_\lambda$}
\psfrag{P0}[cl]{$P(0,0)$}
\psfrag{P1}[cl]{$P(U,0)$}
\psfrag{P2}[cl]{$P(U,V)$}
\psfrag{P3}[cr]{$P(0,V)$}
\includegraphics[height=6cm]{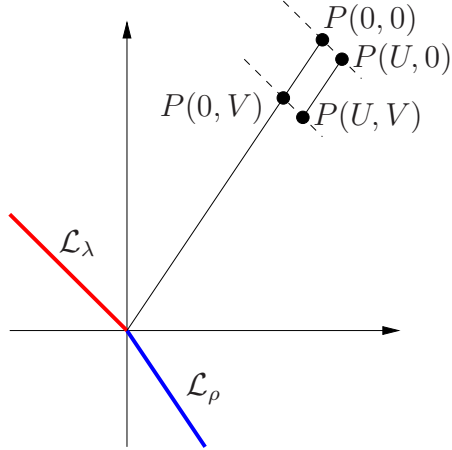}
\caption{Illustration of the different end-points considered in this section. We have $P=P(0)=P(0,0)=((1+\gamma)N,(1-\gamma)N)$, $P(U)=P(U,0)$. In the analysis we take $U,V$ of order $N^{1/3}$.}
\label{FigGeneralP}
\end{center}
\end{figure}
We further define the ``spatial'' rescaled process
\begin{equation}\label{eq:DefZeta1}
\begin{aligned}
 L^{{\rm resc},\lambda}_{N}(u)&:= \frac{L_{{\cal L}_{\lambda}\to P(uN^{1/3})}-\left(\mu_0 N - \frac{u(1-2\lambda)}{(1-\lambda)\lambda}N^{1/3}\right)}{N^{1/3}},\\
 L^{{\rm resc},\rho}_{N}(u) & := \frac{L_{{\cal L}_{\rho}\to P(uN^{1/3})}- \left(\mu_0 N - \frac{u(1-2\rho)}{\rho(1-\rho)}N^{1/3}\right)}{N^{1/3}},
 \end{aligned}
\end{equation}
 where $\gamma,\mu_0$ are same as in Lemma~\ref{Lemma1}.

Throughout the text, we say that two sequences $(Y_{N})_{N\in \N},(\tilde{Y}_{N})_{N\in \N}$ of random variables are \emph{asymptotically independent} to each other,
 if for all $s_1, s_2 \in \R$
 \begin{equation}
 \lim_{N \to \infty}\Pb(Y_N \leq s_1, \tilde{Y}_{N}\leq s_2)= \lim_{N \to \infty}\Pb(Y_N \leq s_1) \lim_{N \to \infty}\Pb(\tilde{Y}_{N}\leq s_2).
 \end{equation}
 \begin{prop}\label{ComponentConvergenceProp}
 Recall the definition of $\mu, \sigma_1$ and $\sigma_2$ from Theorem~\ref{MainTheoremForDifferentDensity}. Then, for any fixed $u\in \mathbb{R}$, we have
 \begin{equation}\label{eq:GOEConvergence1}
 \begin{aligned}
 \lim_{N\to \infty} \Pb\big(L^{{\rm resc},\lambda}_{N}(u)\leq s\big)& = F_{\GOE}\left(\frac{2^{2/3}s}{\sigma_1}\right),\\
 \lim_{N\to \infty} \Pb\big(L^{{\rm resc},\rho}_{N}(u)\leq s\big) & = F_{\GOE}\left(\frac{2^{2/3}s}{\sigma_2}\right).
 \end{aligned}
 \end{equation}
 Moreover, for any fixed $u$, $L^{{\rm resc},\lambda}_{N}(u)$ and $L^{{\rm resc},\rho}_{N}(u)$ are asymptotically independent to each other.

Otherwise stated, for any fixed $u$, $(L^{{\rm resc},\lambda}_{N}(u),L^{{\rm resc},\rho}_{N}(u))$ converges in distribution to $(\xi^{(1)}_{\mathrm{GOE}}\sigma_1 2^{-2/3}, \xi^{(2)}_{\mathrm{GOE}}\sigma_2 2^{-2/3})$ where $\xi^{(k)}_{\rm GOE}$, $k=1,2$, are two independent GOE Tracy-Widom distributed random variables.
 \end{prop}
\begin{proof}
The limiting formulas (\ref{eq:GOEConvergence1}) are obtained by replacing $\gamma\to\gamma+u N^{-2/3}$ in the derivation of Lemma~\ref{Lemma1}. Similarly, Lemma~\ref{Lemma2} holds also with $\gamma\to\gamma+u N^{-2/3}$, so that the end-point of the LPP problem is $P(uN^{1/3})=\left((1+\gamma)N+u N^{1/3}, (1-\gamma)N-uN^{1/3}\right)$.
Concerning the asymptotic independence, it can be obtained  as a consequence of the general Theorem~2.1 in~\cite{FN13}, which holds under three Assumptions given in~\cite{FN13}. Now \eqref{eq:GOEConvergence1} verifies Assumption~1 of~\cite{FN13}. Lemma~\ref{Lemma2} verifies Assumption~2 of~\cite{FN13} (upon replacing $P$ by $P(uN^{1/3})$), which are the hypothesis of the slow-decorrelation theorem, see Theorem~2.1 of~\cite{CFP10b}. This theorem implies that difference of $L^{{\rm resc},\lambda}_{N}(u)$ (or, $L^{{\rm resc},\rho}_{N}(u)$) and random variables defined in (\ref{eq:Lemma4}) (of course making the replacement $\gamma\to\gamma+u N^{-2/3}$) goes to $0$ in probability, as $N\to\infty$. Finally, Lemma~\ref{Lemma3} verifies Assumption~3 of~\cite{FN13}, which says that that the maximizers to $E_\lambda$ and $E_\rho$ stay in two separate sets of independent random variables with probability $1$ as $N\to\infty$. This finishes the proof.
\end{proof}

By KPZ scaling theory, one expects that the processes $u\mapsto L^{{\rm resc},\lambda}_{N}(uN^{1/3})$ (resp, $u\mapsto L^{{\rm resc},\rho}_{N}(uN^{1/3})$) converges to an Airy$_1$ process in $u$ (see~\cite{FO17v1} for a proof of it). This means that if we look at the processes $u\mapsto L^{{\rm resc},\lambda}_{N}(u)$ (or, $ L^{{\rm resc},\lambda}_{N}(u)$), then it will converges to a process with constant value in the $N\to\infty$ limit. The reason why we consider here the latter scaling is that the shock, which is the observable that we want to analyze, has width of order $N^{1/3}$ only. For proving Theorem~\ref{MainTheoremForDifferentDensity} we will need to make this observation precise.

\begin{prop}\label{Tightness}
Fix any $C>0$. Then, for any $\e>0$ , we have
\begin{align}\label{eq:LimitOfModulus}
&\lim_{N\to \infty}\Pb\bigg(\sup_{|u|,|v|\leq C}|L^{{\rm resc},\lambda}_{N}(u)-L^{{\rm resc},\lambda}_{N}(v)|) >\e\bigg)= 0, \\
&\lim_{N\to \infty}\Pb\bigg(\sup_{|u|,|v|\leq C}|L^{{\rm resc},\rho}_{N}(u)-L^{{\rm resc},\rho}_{N}(v)|) >\e\bigg)= 0.
\end{align}
\end{prop}

At time $t$, the second class particle will correspond to a point on the competition interface be at a distance $O(N^{1/3})$ around the position $P$, but with a $N$ not being fixed, rather also fluctuating of order $t^{1/3}$ around a macroscopic value. For that reason we will have to consider increments of the process in a $N^{1/3}$ neighborhood of $P$ in any directions. Thus let us extend the definition of $P$ as follow
\begin{align}\label{eq:ExtendDef}
P(U,V):= \left((1+\gamma)\tilde{N}, (1-\gamma)\tilde{N} \right)+U(1,-1) , \quad \tilde{N}:= N+V,
\end{align}
and we use the notation
\begin{equation}
\bar{P}(u,v):=P(uN^{1/3},vN^{1/3}).
\end{equation}
Accordingly, we extend the definition of the rescaled processes in ``space and time'' as
\begin{equation}\label{eqRescProcessesSpaceTime}
\begin{aligned}
 L^{{\rm resc},\lambda}_{N}(u,v)&:= \frac{L_{{\cal L}_{\lambda}\to \bar P(u,v)}-\left(\mu_0 (N+vN^{1/3}) - \frac{(1-2\lambda)}{(1-\lambda)\lambda}uN^{1/3}\right)}{N^{1/3}},\\
 L^{{\rm resc},\rho}_{N}(u,v) & := \frac{L_{{\cal L}_{\rho}\to \bar P(u,v)}- \left(\mu_0 (N+vN^{1/3}) - \frac{(1-2\rho)}{\rho(1-\rho)}uN^{1/3}\right)}{N^{1/3}},
 \end{aligned}
\end{equation}
where $\gamma,\mu_0$ as before. Below we give statements involving $L^{{\rm resc},\lambda}_{N}$ only. Similar statements hold for $L^{{\rm resc},\rho}_{N}$ as well since the two problems are the same up to exchange of the two axes.

Consider the parallelogram $\mathcal{S}_{C}: =\{P(uN^{1/3},vN^{1/3})|u\in [-C,C], v\in [-C,C]\}$. So far, we have only explored the LPP times along a diagonal of $\mathcal{S}_C$. Here, we claim that similar result holds along any line parallel to any of the diagonals inside the parallelogram $\mathcal{S}_C$.
We have $\max_{u_1,u_2\in [-C,C]}|L_N^{{\rm resc},\lambda}(u_1,0)-L_N^{{\rm resc},\lambda}(u_2,0)|$ converges in probability to $0$ when $N$ goes to $0$ by Proposition~\ref{Tightness}. We claim that this property holds for any other $v\in [-C,C]$. In fact, it is true for the increments in any space-time direction as well. Explicitly, we have the following statements.
 \begin{prop}\label{ConcBoundOnMaximum}
Take any $\delta\in (0,1/3)$. Then there exists $N_0=N_0(C)$ such that for all $N\geq N_0$, the following inequalities
 \begin{align}
 \Pb\left(\max_{u\in [-C,C]}|L_N^{{\rm resc},\lambda}(u,v)-L_N^{{\rm resc},\lambda}(0,v)|\geq \epsilon \right) \leq K_1\exp\left( - K_2\epsilon^2 \frac{N^{1/3-\delta}}{C}\right),\label{eq:ExpectedResult1}\\
 \Pb\left(\max_{v\in [-C,C]}|L_N^{{\rm resc},\lambda}(u,v)-L_N^{{\rm resc},\lambda}(u,0)|\geq \epsilon\right) \leq K_1\exp\left( - K_2\epsilon^2 \frac{N^{1/3-\delta}}{C}\right),\label{eq:ExpectedResult2}
\end{align}
hold where $K_1:=K_1(\lambda,\delta)$, $K_2:=K_2(\lambda,\delta)$ denote two large constants. The constants are uniform for $v\in [-C,C]$, resp.\ $u\in [-C,C]$.
\end{prop}

\begin{cor}\label{UsefulCorBis}
Fix any $C>0$. Then, for any $\epsilon>0$, we have
\begin{align}
\lim_{N\to \infty}\Pb\left(\max_{u,v\in [-C,C]}\left|L_N^{{\rm resc},\lambda}(u,v)-L_N^{{\rm resc},\lambda}(0,0)\right|\geq \epsilon\right) =0.
\end{align}
\end{cor}
\begin{proof}
We first note
\begin{equation}
\begin{aligned}
&\Pb\left(\max_{u,v\in [-C,C]}\left|L_N^{{\rm resc},\lambda}(u,v)-L_N^{{\rm resc},\lambda}(0,0)\right|\geq \epsilon\right)\\
\leq &\sum_{v\in [-C,C]\cap N^{-1/3}\Z} \Pb\left(\max_{u\in [-C,C]}\left|L_N^{{\rm resc},\lambda}(u,v)-L_N^{{\rm resc},\lambda}(0,v)\right|\geq \epsilon/2\right)\\
&+\sum_{v\in [-C,C]\cap N^{-1/3}\Z} \Pb\left(\left|L_N^{{\rm resc},\lambda}(0,v)-L_N^{{\rm resc},\lambda}(0,0)\right|\geq \epsilon/2\right)\\
\end{aligned}
\end{equation}
By virtue of Proposition~\ref{ConcBoundOnMaximum}, the right side of the above inequality is bounded by $4CK_1N^{1/3}\exp(-K_2\epsilon^2 N^{1/3-\delta}/2)$ (for some $0<\delta<1/3$) which converges to $0$ as $N$ goes to $\infty$.
\end{proof}

 \subsection{Proof of Theorem~\ref{MainTheoremForDifferentDensity}}\label{SectProofMainResult}
The number of jumps made by the second class particle until time $t$ is \mbox{$2t/\mu_0+\Or(t^{1/3})$}. For that reason, we introduce the notation $N=t/\mu_0$. Since the trajectory of the second class particle is the same as the trajectory of the competition interface, the second class particle after $2N$ steps will be at position $(I_{2N},J_{2N})=P({\cal U}N^{1/3},0)$ for some ${\cal U}$. First we study ${\cal U}$. This will be the starting point to get the distribution of the second class particle at time $t$.

Observe that, by the definition of the competition interface, for any $u\in \R$, we have
 \begin{equation}
 \begin{aligned}\label{eq:CompetitionProperty}
 \Pb\left(P(uN^{1/3},0\right)\in \Gamma^{\infty}_{-})&\leq \Pb\left(I_{2N}-J_{2N}\leq 2\gamma N +2 u N^{1/3}\right)\\ & \leq \Pb\left(P(uN^{1/3}+1,0)\in \Gamma^{\infty}_{-}\right).
 \end{aligned}
 \end{equation}
Thus the limiting distribution of $I_{2N}-J_{2N}$ is the same as the limit of $\Pb\left(P(uN^{1/3},0\right)\in \Gamma^{\infty}_{-})$. Further, the event $P(uN^{1/3},0)\in \Gamma^{\infty}_{-}$ holds whenever $L_{{\cal L}_{\rho}\to P(uN^{1/3},0)}\geq L_{{\cal L}_{\lambda}\to P(uN^{1/3},0)}$. Rewritten in terms of the rescaled LPP, we have obtained
\begin{equation}
\begin{aligned}
&\lim_{N\to\infty}\Pb\left(I_{2N}-J_{2N}\leq 2\gamma N +2 u N^{1/3}\right) \\
=& \lim_{N\to\infty} \Pb\left(L_N^{{\rm resc},\lambda}(u,0)-\frac{1-2\lambda}{\lambda(1-\lambda)}u\leq L_N^{{\rm resc},\rho}(u,0)-\frac{1-2\rho}{\rho(1-\rho)}u\right),\\
=&\lim_{N\to\infty} \Pb\left(L_N^{{\rm resc},\lambda}(u,0)-L_N^{{\rm resc},\rho}(u,0)\leq u \Upsilon\right),
\end{aligned}
\end{equation}
where we defined
\begin{equation}
\Upsilon=\frac{1-2\lambda}{\lambda(1-\lambda)} - \frac{1-2\rho}{\rho(1-\rho)}>0
\end{equation}
since $\lambda<\rho$.
By Proposition~\ref{ComponentConvergenceProp}, $L_N^{{\rm resc},\lambda}(u,0)$ and $L_N^{{\rm resc},\rho}(u,0)$ converges to $u$-independent non-degenerate distribution functions. Consequently, for any $\e>0$, there exists a finite $C_1=C_1(\e)>0$ such that $\Pb\left(I_{2N}-J_{2N}\geq 2\gamma N +C_1 N^{1/3}\right)\leq \e$ and $\Pb\left(I_{2N}-J_{2N}\leq 2\gamma N -C_1 N^{1/3}\right)\leq \e$ for all $N$ large enough.

Since ${\cal U}$ is characterized by the requirement
\begin{equation}\label{eq5.15}
L_N^{{\rm resc},\lambda}({\cal U},0)-\frac{1-2\lambda}{\lambda(1-\lambda)}{\cal U} =L_N^{{\rm resc},\rho}({\cal U},0)-\frac{1-2\rho}{\rho(1-\rho)}{\cal U}
\end{equation}
up to $\Or(N^{-1/3})$, we have
\begin{equation}\label{eq5.21}
\Pb(|{\cal U}|\geq 2C_1)\leq 2\e
\end{equation}
for all $N$ large enough.

Let us now approximate the distribution of ${\cal U}$ by the rescaled LPP at the fixed point $P(0,0)$. By Proposition~\ref{ComponentConvergenceProp}, the random variables
\begin{equation}\label{eq5.22}
\chi_\lambda:=L_N^{{\rm resc},\lambda}(0,0)\quad\textrm{and}\quad \chi_\rho:=L_N^{{\rm resc},\rho}(0,0)
\end{equation}
have non-degenerate limiting distributions and they are asymptotically independent. By Proposition~\ref{Tightness}, uniformly for $\cal U$ in a bounded interval, for any $\tilde\e>0$,
\begin{equation}\label{eq5.23}
\Pb(|L_N^{{\rm resc},\lambda}({\cal U},0)-\chi_\lambda|\geq \tilde\e)\leq \e\quad\textrm{and}\quad \Pb(|L_N^{{\rm resc},\rho}({\cal U},0)-\chi_\rho|\geq \tilde\e)\leq \e
\end{equation}
for all $N$ large enough. This and (\ref{eq5.21}) imply that, for any $\e'>0$,
\begin{equation}
\Pb\left(\left|{\cal U}-\frac{\chi_\lambda-\chi_\rho}{\Upsilon}\right|\geq \e'\right)\leq 4\e
\end{equation}
for all $N$ large enough. Statement as this will be also written in a compact way as ${\cal U} = \Upsilon^{-1} (\chi_\lambda-\chi_\rho)+o(1)$ where $o(1)$ means random variable which converges in probability to $0$ as $N\to \infty$.

The next quantity we have to control is the difference of the  number of steps to the right and to the left made by the second class particle at time $t$. Let use introduce the random variables $\widetilde {\cal U}$ and $\cal V$ by the relation
\begin{equation}
(I(t),J(t))=P(\widetilde {\cal U} N^{1/3},{\cal V} N^{1/3}).
\end{equation}
First notice that the time when the second class particle made $2N$ steps is given by
\begin{equation}\label{eq5.18}
\tau=\mu_0 N + \left[L_N^{{\rm resc},\lambda}({\cal U},0)-\frac{1-2\lambda}{\lambda(1-\lambda)}{\cal U}\right] N^{1/3}+o(1).
\end{equation}
By (\ref{eq5.21})-(\ref{eq5.23}), there exists a $C_3=C_3(\e)$ such that
\begin{equation}\label{eq5.26}
\Pb\left(\left|\tau- t\right|\geq C_3 N^{1/3}\right)\leq \e
\end{equation}
for all $N$ large enough. The number of jumps of a second class particle is bounded by a Poisson process of intensity $2$ (it has his internal clock, but can be moved by the first class particles, and both have clock rate $1$). Thus, together with (\ref{eq5.26}), there exists a $C_4=C_4(\e)$ such that
\begin{equation}
\Pb(|{\cal V}|\geq C_4)\leq \e
\end{equation}
for all $N$ large enough. Also, since the distance made between $I_{2N}-J_{2N}$ and $I(t)-J(t)$ is bounded by the number of steps by the second class particle during $[\tau,t]$, we also have
\begin{equation}
\Pb(|\widetilde {\cal U}|\geq C_4+2C_1)\leq 3\e
\end{equation}
for all $N$ large enough. Thus there exists a $C=C(\e)$ such that for all $N$ large enough,
\begin{equation}\label{eq:OrderOfShift}
\Pb\left(\max\{|\widetilde{\mathcal{U}}|, |\mathcal{V}|\}\leq C\right)\geq 1-4\e.
\end{equation}

We have $t=L_{{\cal L}_\lambda\to \bar P(\widetilde {\cal U},{\cal V})} =L_{{\cal L}_\rho\to \bar P(\widetilde {\cal U},{\cal V})}$, where equalities hold up to $\Or(1)$.
In terms of scaled variables, it means
\begin{equation}
L^{{\rm resc},\lambda}_N(\widetilde {\cal U},{\cal V})+\mu_0 {\cal V}-\frac{1-2\lambda}{\lambda(1-\lambda)}\widetilde {\cal U}=0.
\end{equation}
Let us prove that, for any $\bar\e>0$,
\begin{equation}\label{eq5.29}
\Pb\left(\left|L^{{\rm resc},\lambda}_N(0,0)+\mu_0{\cal V}-\frac{1-2\lambda}{\lambda(1-\lambda)}\widetilde {\cal U}\right|\geq \bar\e\right)\leq 5\e,
\end{equation}
for any $N$ large enough. This is true, since for any $N$ large enough,
\begin{equation}
\begin{aligned}
&\Pb\left(|L^{{\rm resc},\lambda}_N(\widetilde {\cal U},{\cal V})-L^{{\rm resc},\lambda}_N(0,0)|\geq \bar\e\right)\leq \Pb\left(\max\{|\widetilde{\mathcal{U}}|, |\mathcal{V}|\}\geq C\right)\\
&+\Pb\left(\max_{u,v\in[-C,C]}|L^{{\rm resc},\lambda}_N(u,v)-L^{{\rm resc},\lambda}_N(0,0)|\geq \bar\e\right)\leq 5\e,
\end{aligned}
\end{equation}
where we applied Corollary~\ref{UsefulCorBis} to bound the last term by $\e$.

Recall the definition of $\chi_\lambda$ and $\chi_\rho$ in (\ref{eq5.22}). Then (\ref{eq5.29}) (and the same statement for $\lambda$ replaced by $\rho$) writes
\begin{equation}
\begin{aligned}
\chi_\lambda+o(1)& = \frac{1-2\lambda}{\lambda(1-\lambda)}\widetilde {\cal U}-\mu_0 {\cal V},\\
\chi_\rho+o(1)& = \frac{1-2\rho}{\rho(1-\rho)}\widetilde {\cal U}-\mu_0 {\cal V}.
\end{aligned}
\end{equation}
 This gives
\begin{equation}
\begin{aligned}
\widetilde {\cal U}&=\frac{\chi_\lambda-\chi_\rho}{\Upsilon}+o(1),\\
{\cal V}&=\frac{1}{\mu_0 \Upsilon}\left(\frac{1-2\rho}{\rho(1-\rho)}\chi_\lambda-\frac{1-2\lambda}{\lambda(1-\lambda)}\chi_\rho\right)+o(1).
\end{aligned}
\end{equation}
Recalling that $X_t=I(t)-J(t)=2\gamma N + 2(\gamma {\cal V}+\widetilde {\cal U}) N^{1/3}$, we finally get
\begin{equation}
\begin{aligned}
\frac{X_t-\mathbf{v}t}{t^{1/3}} &= \mu_0^{-1/3} 2(\widetilde {\cal U}+\gamma {\cal V})+o(1)\\
&=\frac{2}{\mu_0^{4/3}\Upsilon}\left(\frac{\chi_\lambda}{\rho(1-\rho)}-\frac{\chi_\rho}{\lambda(1-\lambda)}\right)+o(1).
\end{aligned}
\end{equation}
By Proposition~\ref{ComponentConvergenceProp}, as $N\to\infty$, $\chi_\lambda \to \sigma_1 2^{-2/3}\xi_{\rm GOE}^{(1)}$ and $\chi_\rho \to \sigma_2 2^{-2/3}\xi_{\rm GOE}^{(2)}$, with $\xi_{\rm GOE}^{(i)}$, $i=1,2$, being GOE Tracy-Widom distributed random variables~\cite{TW96} and asymptotically independent. The result (\ref{eq:NumberoOfJumpsDistribution}) follows in the same way from the equality $N_t=P(\widetilde {\cal U}N^{1/3},{\cal V}N^{1/3})=2N+2{\cal V} N^{1/3}$. This ends the proof of the main theorem.

\subsection{Proofs of the preliminary results}\label{AuxilliaryResults}
First we deal with Proposition~\ref{Tightness}. We prove the result only in the case of ${\cal L}_\lambda$, the case for ${\cal L}_\rho$ is proven in the same way. The key ingredient is the coupling between the last passage percolation with two different initial profiles  as in Lemma~2.1 of~\cite{Pi17}, one of which being stationary.

\subsubsection{Comparison with stationary LPP}
Recall that $P=\left((1+\gamma)N, (1-\gamma)N\right)$. As already discussed
\begin{equation}
\argmax_{k\geq 0}\left\{L_{\left(\lambda^{-1}(\lambda-1) k, k\right)\to P}\right\}
\end{equation}
lies in the neighborhood of the point $\xi_\lambda$ defined in (\ref{eqB}).

Using Lemma~4.2 of~\cite{BCS06} one can define a stationary model where the starting line of the LPP is any right/down path, obtained by adding appropriate random weights to the points on the line. Let us define it for
\begin{equation}
{\cal L}^\infty_\lambda=\left\{\left(\left\lfloor\tfrac{\lambda-1}{\lambda} x\right\rfloor, x\right)\big| x\in \Z\right\}.
\end{equation}
Fix any positive real number $\varrho\in (0,1)$. Let us consider a family of independent random variables
\begin{equation}
\{p_i,q_i,i\in\Z\},\quad p_i\sim {\rm Exp}(1-\varrho),\ q_i\sim {\rm Exp}(\varrho).
\end{equation}
Then define the weights $\{\omega^{\varrho}(k)\big| k\in \Z\}$ given by
\begin{align}\label{eq:RandomWeightsAlongAntiDiagonal}
\omega^{\varrho}(k):= \left\{\begin{array}{ll}
\sum_{i=\lfloor (\lambda-1) k/\lambda\rfloor+1}^0 p_i-\sum_{i=k+1}^0 q_i & \text{for } k<0,\\
0 & \text{for } k=0,\\
-\sum_{i=1}^{\lfloor (\lambda-1) k/\lambda\rfloor} p_i+\sum_{i=1}^k q_i & \text{for } k>0.
\end{array}\right.
\end{align}
Using these weights, we define the stationary last passage time $L^{{\rm stat},\varrho}$ from ${\cal L}^\infty_\lambda$ in the neighborhood of the point $P$ as follows
\begin{align}\label{eq:LastPassageTimeFromLHalf}
L^{{\rm stat},\varrho}_{{\cal L}^\infty_\lambda\to P(x)} := \max_{k\in\Z}\{L_{\left(\lambda^{-1}(\lambda-1)k,k\right)\to P(x)}+\omega^{\varrho}(k)\}.
\end{align}
The name stationary follows from the fact, by Lemma~4.2 of~\cite{BCS06}, for any $m,n\geq 0$ we have the property
\begin{equation}\label{eq:StatProperty}
k\mapsto L^{{\rm stat},\varrho}_{{\cal L}^\infty_\lambda\to (m,n)}-L^{{\rm stat},\varrho}_{{\cal L}^\infty_\lambda\to (m+k,n)}\sim \sum_{l=1}^k \zeta_l,
\end{equation}
with $\zeta_l$ i.i.d.\ ${\rm Exp}(1-\lambda)$ random variables. Furthermore, we denote the exit point of the last passage path associated with $L^{{\rm stat},\varrho}$ from the line ${\cal L}^\infty_\lambda$ to the point $P(x)$ by $Z^{{\rm stat},\varrho}(x)$,
more precisely,
\begin{equation}
Z^{{\rm stat},\varrho}(x)=\argmax_{k\in\Z}\{L_{\left(\lambda^{-1}(\lambda-1)k,k\right)\to P(x)}+\omega^{\varrho}(k)\}.
\end{equation}
From stationarity we have the following translation-invariance property
\begin{equation}\label{eqTranslInv}
Z^{{\rm stat},\varrho}(x) \stackrel{d}{=} Z^{{\rm stat},\varrho}(0)+x.
\end{equation}
Similarly, we define
\begin{equation}\label{eq5.37}
Z_\lambda(x)=\argmax_{k\geq 0}\{L_{\left(\lambda^{-1}(\lambda-1)k,k\right)\to P(x)}\}.
\end{equation}
The characteristic direction of the last passage percolation $L^{{\rm stat},\varrho}$ is given by $\left((1-\varrho)^2, \varrho^2\right)$. The line joining the point
\begin{equation}
Q_{N}(\xi_{\lambda}):=((\lambda-1)\lambda^{-1} \xi_{\lambda}N, \xi_{\lambda}N)
\end{equation} and $P$ has direction $((1-\lambda)^2, \lambda^2)$. Hence, the characteristic line in the stationary LPP $L^{{\rm stat},\lambda}_{{\cal L}^\infty_\lambda\to P}$ leaves ${\cal L}^\infty_\lambda$ at $Q_{N}(\xi_{\lambda})\in {\cal L}_\lambda$.

Now, we briefly describe how we complete the proof of Proposition~\ref{Tightness}. Explicit arguments will be given towards the end of this section. For $C,r\in \R$, we define
\begin{equation}
\lambda_{\pm}(r) := \lambda \pm \frac{r}{N^{1/3}},
\end{equation}
and the good event,
\begin{equation}\label{eq:GoodEvent}
G_{N}(r):= \left\{Z^{{\rm stat},\lambda_-}(CN^{1/3})\leq Z_\lambda(0)\right\} \cap \{Z_\lambda(CN^{1/3})\leq Z^{{\rm stat},\lambda_+}(0)\}.
\end{equation}
One can bound $\Pb(\Omega^N_{\lambda}[-C,C]\geq \epsilon)$ by the sum of $\Pb((G_N(r))^c)$ and $\Pb(G_N(r)\cap \{\Omega^{N}_{\lambda}[-C,C]\geq \epsilon\})$. We show that $\Pb(G_N(r)\cap \{\Omega^{N}_{\lambda}[-C,C]\geq \epsilon\})$ goes to $0$ as $N$ increases using Proposition~\ref{SandwithLemma} which leans on the coupling inequalities of Lemma~\ref{CouplingLemma}. To show $\Pb((G_N(r))^c)$ converges to $0$ as $r$ increases, we use Proposition~\ref{CrucialEventIdentification} which will be proved using Lemma~\ref{AlternativeExitPointLemma} and Lemma~\ref{ExitPointOfGeneralPercolation}.

First lemma corresponds to a coupling between the last passage percolation with boundary conditions $\omega$ and the ones without them. This is an extension of Lemma~2.1 of~\cite{Pi17} to the case of general $\varrho$.
\begin{lem}\label{CouplingLemma}
Fix some $0<x_1<x_2$. Now, whenever $Z^{{\rm stat},\varrho}(x_1)\leq Z_\lambda(x_2)$, then we get
\begin{align}\label{eq:UpsideCoupling}
L_{{\cal L}_{\lambda}\to P(x_2)} - L_{{\cal L}_{\lambda}\to P(x_1)}\leq L^{{\rm stat},\varrho}_{{\cal L}_{\lambda}^\infty \to P(x_2)} - L^{{\rm stat},\varrho}_{{\cal L}_{\lambda}^\infty\to P(x_1)}.
\end{align}
In contrast, when $Z_{\lambda}(x_1)\leq Z^{{\rm stat},\varrho}(x_2)$ holds, then we have
\begin{align}\label{eq:DownsideCoupling}
 L^{{\rm stat},\varrho}_{{\cal L}_{\lambda}^\infty \to P(x_2)} - L^{{\rm stat},\varrho}_{{\cal L}_{\lambda}^\infty\to P(x_1)}\leq L_{{\cal L}_{\lambda}\to P(x_2)} - L_{{\cal L}_{\lambda}\to P(x_1)}.
\end{align}
\end{lem}
\begin{proof}
One can find similar result in Lemma~1 of~\cite{CP15b} for the exit point in Hammerseley process. Furthermore,~\cite{Pi17} it has been generalized in the case of line to point last passage percolation with any two initial profiles along the anti-diagonal. Our proof closely follows the arguments of Lemma~2.1 of~\cite{Pi17}. We prove in detail only \eqref{eq:UpsideCoupling} since the proof of \eqref{eq:DownsideCoupling} is similar.

We denote $z_2= Z_\lambda(x_2)$ and $z_1=Z^{{\rm stat},\varrho}(x_1)$, see Figure~\ref{FigCrossing} for an illustration.
\begin{figure}
\begin{center}
\psfrag{P1}[cl]{$P(x_1)$}
\psfrag{P2}[cl]{$P(x_2)$}
\psfrag{Z1}[cr]{$z_2=Z_\lambda(x_2)$}
\psfrag{Z2}[cr]{$z_1=Z^{{\rm stat},\varrho}(x_1)$}
\psfrag{c}[cr]{$\mathbf{c}_\lambda$}
\psfrag{L}[cl]{${\cal L}_\lambda^\infty$}
\includegraphics[height=6cm]{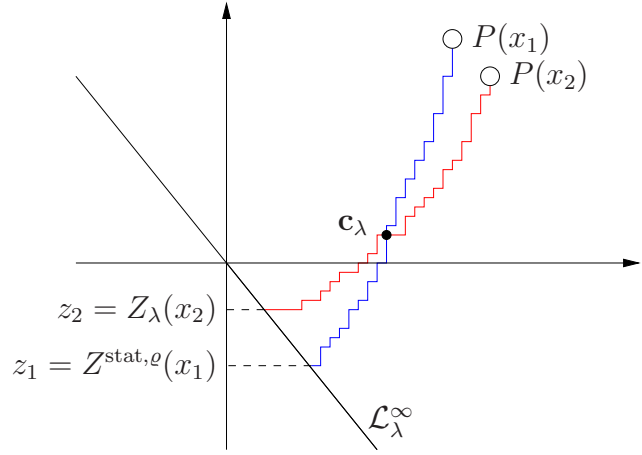}
\caption{Illustration of the geometric setting of Lemma~\ref{CouplingLemma}.}
\label{FigCrossing}
\end{center}
\end{figure}
Under the constraint $z_1\leq z_2$, the last passage path from $\left(\lambda^{-1}(\lambda-1)z_1,z_1\right)$ to $P(x_1)$ crosses somewhere in between the
last passage path from $(\lambda^{-1}(\lambda-1)z_2,z_2)$ to $P(x_2)$. We call $\mathbf{c}_{\lambda}$ the point where two paths meet.
Thus we have
\begin{align}
L^{{\rm stat},\varrho}_{{\cal L}_{\lambda}^\infty\to P(x_1)}&= \omega^{\varrho}(z_1)+L_{z_1\to \mathbf{c}_{\lambda}} +L_{\mathbf{c}_{\lambda}\to P(x_1)},\label{eq:PathRelationParseSubAdditivity1}\\
L_{{\cal L}_{\lambda}\to P(x_2)} &= L_{z_2\to \mathbf{c}_{\lambda}}+ L_{\mathbf{c}_{\lambda}\to P(x_2)},\label{eq:PathRelationParseSubAdditivity4}
\end{align}
and, by subadditivity of the last passage time, we have
\begin{align}
L^{{\rm stat},\varrho}_{{\cal L}_{\lambda}^\infty\to P(x_2)} &\geq \omega^{\varrho}(z_1)+L_{z_1\to \mathbf{c}_{\lambda}}+ L_{\mathbf{c}_{\lambda}\to P(x_2)},\label{eq:PathRelationParseSubAdditivity2}\\
L_{{\cal L}_{\lambda}\to P(x_1)}&\geq L_{z_2\to \mathbf{c}_{\lambda}} +L_{\mathbf{c}_{\lambda}\to P(x_1)}.\label{eq:PathRelationParseSubAdditivity3}
\end{align}
Subtracting \eqref{eq:PathRelationParseSubAdditivity3} from \eqref{eq:PathRelationParseSubAdditivity4}, we get
\begin{align}\label{eq:OneSideoFinequality}
L_{{\cal L}_{\lambda}\to P(x_2)}-L_{{\cal L}_{\lambda}\to P(x_1)}\leq L_{\mathbf{c}_{\lambda}\to P(x_1)}-L_{\mathbf{c}_{\lambda}\to P(x_2)},
\end{align}
while subtracting \eqref{eq:PathRelationParseSubAdditivity1} from \eqref{eq:PathRelationParseSubAdditivity2}, one obtains
\begin{align}\label{eq:OtherSideOfTheInequality}
L^{{\rm stat},\varrho}_{{\cal L}_{\lambda}^\infty\to P(x_2)}-L^{{\rm stat},\varrho}_{{\cal L}_{\lambda}^\infty\to P(x_1)}\geq L_{\mathbf{c}_{\lambda}\to P(x_1)}-L_{\mathbf{c}_{\lambda}\to P(x_2)}.
\end{align}
\eqref{eq:UpsideCoupling} follows by combining \eqref{eq:OneSideoFinequality} and \eqref{eq:OtherSideOfTheInequality}.
\end{proof}

\subsubsection{Bounds on exit points probability}
In what follows, we state a result for the exit point $Z^{{\rm stat},\varrho}$ of the stationary last passage percolation. A weaker bound with the restriction on $\varrho\in [1/4,3/4]$ can be found in Lemma~2.2 of~\cite{Pi17}.
\begin{lem}\label{AlternativeExitPointLemma}
Fix any $\varrho\in (0,1)$ and $c\in\R_{>0}$. Consider the stationary last passage percolation model $L^{{\rm stat},\varrho}_{\mathcal{L}^{\infty}_{\lambda}}$. Denote the exit point of the last passage path from $\mathcal{L}^{\infty}_{\lambda}$ to $P^{\varrho}(x):=((1-\varrho)^2 N+x, \varrho^2 N-x)$ by $\widetilde Z^{{\rm stat},\varrho}(x)$. Then, there exists $M_0=M_0(c,\lambda,\varrho)$ such that for all $M\geq M_0$, we have
\begin{align}\label{eq:BoundOnTheExitPoint}
\Pb\left(\left|\widetilde Z^{{\rm stat},\varrho}(c N^{1/3})\right|\geq M N^{2/3}\right)\leq C_1 e^{- C_2M^2},
\end{align}
uniformly for all large $N$ where $C_1:= C_1( \lambda, \varrho)$ and $C_2:=C_2(\lambda,\varrho)$ are two constants.
\end{lem}
\begin{proof}
We only show the bound on the upper tail. Proof of the lower tail follows by the same arguments.
For the proof we need to divide the maximizer on several pieces. For that purpose, in this proof we denote by
\begin{equation}
L_{a}(x):=\max_{k\geq a} \{L_{(\lambda^{-1}(\lambda-1) k,k)\to P^\varrho(x)}+\omega^\varrho(k)\}.
\end{equation}

First notice that
\begin{equation}\label{eq5.55}
\begin{aligned}
\Pb\left(\widetilde Z^{{\rm stat},\varrho}>MN^{2/3}\right)&\leq \Pb\left(L_{M N^{2/3}}(c N^{1/3}) \geq L_{(0,0)\to P^\varrho(c N^{1/3})}\right),\\
&\leq \Pb\left(L_{M N^{2/3}}(c N^{1/3})\geq s\right)+\Pb\left(L_{(0,0)\to P^\varrho(c N^{1/3})}<s\right),
\end{aligned}
\end{equation}
for any choice of $s$. The law of large number estimates are
\begin{equation}\label{eq4.56}
\begin{aligned}
&L_{(0,0)\to P^\varrho(c N^{1/3})} \simeq N-\alpha_1 cN^{1/3}, \quad \omega^\varrho(MN^{2/3})\simeq -\alpha_2 M N^{2/3}\\
&L_{(\lambda^{-1}(\lambda-1) M N^{2/3},M N^{2/3})\to P^\varrho(cN^{1/3})} \simeq N + \alpha_2 M N^{2/3}-(\alpha_1 c +\alpha_3 M^2)N^{1/3},
\end{aligned}
\end{equation}
with $\alpha_1=\frac{1-2\varrho}{\varrho(1-\varrho)}$, $\alpha_2=\frac{\varrho-\lambda}{\lambda \varrho (1-\varrho)}$, and $\alpha_3=\frac{(\lambda-2\lambda\varrho+\varrho^2)^2}{4\lambda^2(1-\varrho)^3\varrho^3}$. Due to (\ref{eq4.56}), we choose $s=N-\alpha_1 c N^{1/3}-\frac14\alpha_3 M^2 N^{1/3}$. In this way, we will be able to make both probabilities in (\ref{eq5.55}) very small as $M\to\infty$, uniformly in $N$.

Proposition~\ref{ppLowTail} gives immediately $\Pb\left(L_{(0,0)\to P^\varrho(c N^{1/3})}<s\right)\leq C_1 e^{-C_2 M^2}$. Next, consider the first probability in the r.h.s.~of (\ref{eq5.55}). For $a>0$, use the short-hand notation $\ell(k)=\lambda^{-1}(\lambda-1) k$, we first rewrite $L_a(x)$ as follows,
\begin{equation}\label{eq5.56}
\begin{aligned}
L_{a}(x)&=\max_{k\geq a} \left\{\max_{l\geq k}L_{(\ell(k),k)\to (\ell(a),l)}+\omega^\varrho(k)+L_{(\ell(a),l)\to P^\varrho(x)}\right\} \\
&=\omega^\varrho(a)+\max_{l\geq a} \left\{L_{(\ell(a),l)\to P^\varrho(x)} + \max_{a\leq k \leq l}\{L_{(\ell(k),k)\to (\ell(a),l)}+\omega^\varrho(k)-\omega^\varrho(a)\}\right\}\\
&=\omega^\varrho(a)+ R_a(x)
\end{aligned}
\end{equation}
where
\begin{equation}\label{eq5.57}
\begin{aligned}
R_a(x)=\max_{l\geq a} \bigg\{L_{(\ell(a),l)\to P^\varrho(x)} + \sum_{k=a+1}^l \theta_k\bigg\},
\end{aligned}
\end{equation}
with $\theta_k$'s being i.i.d.\ ${\rm Exp}(\varrho)$ random variables. The equality (\ref{eq5.57}) follows from Lemma~4.1 and Lemma~4.2 of~\cite{BCS06}.

We also have, for any $\tilde s$, the inequality
\begin{equation}\label{eq4.58}
\begin{aligned}
\Pb\left(L_{M N^{2/3}}(c N^{1/3})\geq s\right) &=\Pb\left(\omega^\varrho(M N^{2/3})+R_{M N^{2/3}}(c N^{1/3})\geq s\right) \\
&\leq \Pb\left(\omega^\varrho(M N^{2/3})\geq \tilde s\right) +\Pb\left(R_{M N^{2/3}}(c N^{1/3})\geq s-\tilde s\right).
\end{aligned}
\end{equation}
Recall that $s=N-\alpha_1 c N^{1/3}-\frac14\alpha_3 M^2 N^{1/3}$. Thus choosing $\tilde s=-\alpha_2 M N^{2/3}+\frac14\alpha_3 M^2 N^{1/3}$ we have $\bar s=s-\tilde s = N +\alpha_2 M N^{2/3}-(\alpha_1 c+\frac12\alpha_3 M^2)N^{1/3}$.

Let us bound the first term in (\ref{eq4.58}). Denote by $\tilde q_i=q_i-1/\varrho$ and $\tilde p_i = -p_i+1/(1-\varrho)$. Then, using the exponential Chebyshev's inequality, we have
\begin{equation}\label{eq4.59}
\begin{aligned}
 \Pb\left(\omega^\varrho(M N^{2/3})\geq \tilde s\right) &= \Pb\Bigg(\sum_{i=1}^{M N^{2/3}} \tilde q_i + \sum_{j=1}^{\lambda^{-1}(1-\lambda)M N^{2/3}} \tilde p_j \geq\hat s\Bigg)\\
 &\leq e^{-t\hat s} \E(e^{t\tilde q_1})^{M N^{2/3}} \E(e^{t\tilde p_1})^{\lambda^{-1}(1-\lambda)M N^{2/3}}.
\end{aligned}
\end{equation}
with $\hat s = \frac14 \alpha_3 M^2 N^{1/3}$ for any $t\geq 0$. Taking $t=\frac{M (\lambda-2\lambda\varrho+\varrho^2)}{16\lambda\varrho(1-\varrho)} N^{-1/3}$, after simple calculations we get, for all $N$ large enough,
\begin{equation}
(\ref{eq4.59})\leq e^{-\alpha_4 M^3},
\end{equation}
with $\alpha_4=\frac{(\lambda-2\lambda\varrho+\varrho^2)^3}{4 \lambda^3(4\varrho(1-\varrho))^4}$.

It remains to bound the second term in (\ref{eq4.58}). After shifting the half-line to start at the origin, $R_{MN^{2/3}}(c N^{1/3})$ becomes the LPP from $(0,0)$ to
\mbox{$(\gamma^2 n,n)=(1-\varrho)^2 N+c N^{1/3}+\lambda^{-1}(1-\lambda)MN^{2/3},\varrho^2 N-cN^{1/3}-M N^{2/3})$} with one-sided vertical boundary sources with parameter $\varrho$. This LPP was studied in~\cite{BBP06} in terms of sample covariance matrices (see their Section~6 for the connection to LPP). In their proof they obtained exponential bounds used then to show the convergence to the limiting law. These bounds implies exponential bounds for the right tail of the distribution. An explicit statement in the LPP language of this bound can be found in Lemma~3.3 of~\cite{FO17v1}. In our case it corresponds to have a \mbox{$\kappa=\frac{(\lambda-2\lambda\varrho+\varrho^2)M}{2\lambda\varrho(1-\varrho)}$}. Using (3.8) and the first bound of Lemma~3.3 of~\cite{FO17v1} with the choice $a=\frac{1}{4(1-\varrho)^{2/3}}+O(1/M^2)$ (to match with $\bar s$), we get uniformly for large $N$,
\begin{equation}
\Pb\left(R_{M N^{2/3}}(c N^{1/3})\geq s-\tilde s\right) \leq C e^{-cM^2},
\end{equation}
for all $M\geq M_0$, $M_0$ a constant and $C,c$ some $M$-independent constants. This ends the proof of (\ref{eq:BoundOnTheExitPoint}).
\end{proof}

Below, we note down a simple corollary of Lemma~\ref{AlternativeExitPointLemma}.
\begin{cor}\label{ExitPCor}
Recall the definition of $P(\cdot)$ given in \eqref{eq:defPU}. Consider the last passage percolation model $L^{{\rm stat}, \varrho}_{{\cal L}^{\infty}_{\lambda}}$ for $\varrho=\lambda_{\pm}$. Denote the exit points the corresponding last passage paths from the line ${\cal L}^{\infty}_{\lambda}$ to the point $P(x)$ by $Z^{{\rm stat}, \lambda_\pm}(x)$. Fix some \mbox{$C \in \R$}. Let $\xi^{\pm}_{\lambda}$ be such that the line joining the points \mbox{$Q_N(\xi^{\pm}_{\lambda}):=(\lambda^{-1}(\lambda-1)\xi^{\pm}_{\lambda} N, \xi^{\pm}_{\lambda} N)$} and $P(CN^{1/3})$ has direction $((1-\lambda_{\pm})^2, (\lambda_{\pm})^2)$. Then, there exists $M_0=M_0(\lambda)$, $K_1=K_1(\lambda)$ and $K_2=K_2(\lambda)$ such that for all $M\geq M_0$, the followings hold uniformly for all large $N$
 \begin{align}
 &\Pb\left(Z^{{\rm stat},\lambda_{+}}(CN^{1/3}) - \xi^{+}_{\lambda} N< -M N^{2/3}\right)\leq K_1\exp(-K_2 M^2),\label{eq:ExitLowerTail}\\
 &\Pb\left(Z^{{\rm stat},\lambda_{-}}(CN^{1/3}) - \xi^{-}_{\lambda} N> M N^{2/3}\right)\leq K_1\exp(-K_2 M^2).\label{eq:ExitUpperTail}
 \end{align}
\end{cor}
\begin{proof}
To see this, we first translate $X$-axis and $Y$-axis linearly so that the origin is moved to the points $Q_N(\xi^{\pm}_{\lambda})$. Notice that the definition of ${\cal L}_{\lambda}^\infty$ is invariant under these shifts of the origin. In these new co-ordinate systems, we denote the point $P(CN^{1/3})$ by $\widetilde{P}_{\pm}(CN^{1/3})$ and the corresponding exit points by $\widetilde{Z}^{{\rm stat},\lambda_{\pm}}$. Using the fact $Z^{{\rm stat}, \lambda_\pm}(CN^{1/3}) - \xi^{\pm}_{\lambda} N=\widetilde{Z}^{{\rm stat},\lambda_{\pm}}$, we obtain
 \begin{equation}\label{eq:ExitPointIdentity1}
 \begin{aligned}
 &\Pb\left(Z^{{\rm stat},\lambda_{+}}(CN^{1/3}) - \xi^{+}_{\lambda} N< -M N^{2/3}\right)=\Pb(\widetilde{Z}^{{\rm stat},\lambda_{+}}<-MN^{2/3}),\\
 &\Pb\left(Z^{{\rm stat},\lambda_{-}}(CN^{1/3}) - \xi^{-}_{\lambda} N> M N^{2/3}\right)=\Pb(\widetilde{Z}^{{\rm stat},\lambda_{-}}>MN^{2/3}).
 \end{aligned}
 \end{equation}
Further, we have $\widetilde{P}_{\pm}(CN^{1/3})=t((1-\lambda_{\pm})^2 N, (\lambda_{\pm})^2 N)$ for some constant $t>0$.

From Lemma~\ref{AlternativeExitPointLemma} applied for $\varrho=\lambda_\pm$, there exists $M_0=M_0(\lambda)$, $K_1=K_1(\lambda)$ and $K_2=K_2(\lambda)$ such that for all $M\geq M_0$, we have
\begin{equation}
\Pb\left(|\widetilde{Z}^{{\rm stat},\lambda_{\pm}}|>MN^{2/3}\right)\leq K_1 e^{-K_2M^2},\label{eq:ExitPointTail}
\end{equation}
for all large $N$. Since the probabilities in \eqref{eq:ExitPointIdentity1} are bounded by the ones in \eqref{eq:ExitPointTail}, the proof is completed.
\end{proof}

Concerning the exit point from ${\cal L}_\lambda$, a simple consequence of Lemma~4.3 of~\cite{FO17v1} results into the following estimate.
\begin{lem}\label{ExitPointOfGeneralPercolation}
Let $\xi_\lambda$ be given as in (\ref{eqB}). Fix any $C\in \R$. Then, there exists $M_0$ such that for all $M>M_0=M_0(\lambda)$, we have
 \begin{align}\label{eq:ExitPointProbabilityDecay}
 \Pb\left(\left|Z_\lambda(CN^{1/3})-\xi_\lambda N\right|\geq MN^{2/3}\right) \leq K_1e^{-K_2M^2}
\end{align}
for some constants $K_1=K_1(\lambda)$ and $K_2=K_2(\lambda)$, uniformly for all $N$ large enough.
\end{lem}
\begin{proof}
This results is a simple consequence of Lemma~4.3 of~\cite{FO17v1}. One shifts the origin to get the end-point $((1+\gamma)N+cN^{1/3},(1-\gamma)N-c N^{1/3})$ along the diagonal of the first quadrant as in the proof of Lemma~\ref{Lemma1}. Then one can extend the line ${\cal L}_\lambda$ to ${\cal L}_\lambda^\infty =\left\{\left(\left\lfloor\tfrac{\lambda-1}{\lambda} x\right\rfloor, x\right)\big| x\in \Z\right\}$. Call $Z^{\infty}_{\lambda}$ the exit point position from ${\cal L}_\lambda^\infty$. Then
\begin{equation}
\Pb\left(\left|Z_\lambda(CN^{1/3})-\xi_\lambda N\right|\geq MN^{2/3}\right) \leq \Pb\left(\left|Z^{\infty}_{\lambda}(CN^{1/3})-\xi_\lambda N\right|\geq MN^{2/3}\right).
\end{equation}
In Lemma~4.3 of~\cite{FO17v1} we have exactly a Gaussian estimate (in the variable $M$) of the latter, which finishes the proof.
\end{proof}

\subsubsection{Bound on the probability of good event $G_N(r)$}
Now we generalize Lemma~2.3 of~\cite{Pi17} to give some upper bound on the probability $\Pb((G_{N}(r))^c)$ (see \eqref{eq:GoodEvent} for its definition).
\begin{prop}\label{CrucialEventIdentification}
Fix some $C\in \mathbb{R}$. Denote the exit point of the stationary LPP $L^{\lambda_{\pm}}_{{\cal L}_{\lambda}^{\infty}}$ (with weights $\omega^{\lambda_{\pm}}$) from the line $\mathcal{L}^{\infty}_{\lambda}$ to $P(x)$ by $Z^{{\rm stat},\lambda_{\pm}}(x)$. Adopt all other notations introduced before. Then, there exists $r_0=r_0(C,\lambda)$ such that for all $r\geq r_0$, we have
\begin{align}\label{eq:EventGoesToZero}
\Pb\left( \left(G_{N}(r)\right)^c \right) \leq K_1e^{-K_2r^2}
\end{align}
uniformly for all large $N$ where $K_1=K_1(\lambda)$ and $K_2=K_2(\lambda)$ are two large constants.
\end{prop}
\begin{proof} In the following discussion, we will prove that there exists $r_0$ (depending on $C$ and $\lambda$) such that for all $r\geq r_0$ the following holds (for all large $N$)
\begin{align}
 &\Pb(Z^{{\rm stat},\lambda_{-}}(CN^{1/3}) >Z_{\lambda}(0))\leq 2 K^{\prime}_1e^{-K^{\prime}_2 r^2},\label{eq:EventOneOnlyHere}\\
 &\Pb\left( Z_\lambda(CN^{1/3})> Z^{{\rm stat},\lambda_+}(0)\right)\leq 2 \tilde{K}^{\prime}_1e^{-\tilde{K}^{\prime}_2 r^2}\label{eq:EventTwoHere}.
\end{align}
for two sets of constants $\{K^{\prime}_1,K^{\prime}_2\}$ and $\{\tilde{K}^{\prime}_1,\tilde{K}^{\prime}_2\}$ which depend on $\lambda$. Then \eqref{eq:EventOneOnlyHere}-\eqref{eq:EventTwoHere} implies \eqref{eq:EventGoesToZero}.

\textbf{Proof of \eqref{eq:EventOneOnlyHere}}.
We begin by noting that for any $c_1$ following holds
\begin{equation}\label{eq:DivingTheProbability}
\begin{aligned}
\Pb\left(Z^{{\rm stat},\lambda_-}(CN^{1/3})> Z_\lambda(0)\right) &\leq \Pb\left(Z^{{\rm stat},\lambda_-}(CN^{1/3})-\xi_\lambda N>-c_1N^{2/3}\right)\\
&+ \Pb\left(Z_\lambda(0)-\xi_\lambda N<-c_1N^{2/3}\right).
\end{aligned}
\end{equation}
We have to choose $c_1$ such that both terms on the rhs.\ of (\ref{eq:DivingTheProbability}) are bounded by $K^{\prime}_1 e^{-K^{\prime}_2 r^2}$ for all $r\geq r_1=r_0(C,\lambda)$ uniformly for all large $N$.

The translation invariance property (\ref{eqTranslInv}) gives
\begin{equation}
Z^{{\rm stat},\lambda_{-}}(CN^{1/3}) = Z^{{\rm stat},\lambda_{-}}(0)+CN^{1/3}.
\end{equation}
The characteristic line passing by $((1-\lambda)^2N,\lambda^2 N)$ having direction \mbox{$((1-\lambda_{-})^2, (\lambda_{-})^2)$} intersects the line ${\cal L}^{\infty}_{\lambda}$ at $((1-\lambda)\lambda^{-1} a(r,N),a(r,N))$ with $a(r,N)=a_1rN^{1/3}+a_2 r^2N^{1/3}$ for some constants $a_1(\lambda)>0$, $a_2(\lambda)<0$. Therefore, $Z^{{\rm stat},\lambda_{-}}(0)$ will be close to the point $\xi^{-}_{\lambda}N= \xi_{\lambda}N-(a_1rN^{2/3}+a_2r^2N^{1/3})$. Consequently we can write
\begin{equation}\label{eq:ZVarrhoConc}
\begin{aligned}
&\Pb\left(Z^{{\rm stat},\lambda_{-}}(CN^{1/3})-\xi_{\lambda}N >- c_1N^{2/3}\right)\\=& \Pb\left(Z^{{\rm stat},\lambda_{-}}(0) -\xi^{-}_{\lambda}N>(a_1r-c_1)N^{2/3}-(C+a_2r^2)N^{1/3}\right).
\end{aligned}
\end{equation}
Choose $c_1=a_1r/2$. Then for $N$ large enough,
\begin{equation}\label{eq5.48}
\eqref{eq:ZVarrhoConc}\leq \Pb\left(Z^{{\rm stat},\lambda_{-}}(0)- \xi^{-}_{\lambda}N >a_1rN^{2/3}/4\right)\leq K_1'e^{-K_2'r^{2}}
\end{equation}
by \eqref{eq:ExitUpperTail} of Corollary~\ref{ExitPCor}. On the other side, by Lemma~\ref{ExitPointOfGeneralPercolation} there exists $r_0=r_0(\lambda)$ such that for all $r\geq r_0$, it holds
\begin{equation}\label{eq5.49}
\Pb\left(Z_{\lambda}(0)- \xi_\lambda N<-a_1r N^{2/3}/2\right)\leq K^{\prime}_1e^{-K^{\prime}_2 r^2}.
\end{equation}
Combining (\ref{eq5.48}) and (\ref{eq5.49}) we conclude \eqref{eq:EventOneOnlyHere}.

\textbf{Proof of \eqref{eq:EventTwoHere}}. For any $c_3$, it holds
\begin{align}\label{eq:Dividingprobability2}
\Pb \left(Z_\lambda(CN^{1/3})> Z^{{\rm stat},\lambda_+}(0)\right) &\leq \Pb\left(Z_\lambda(CN^{1/3})-\xi_\lambda N>c_3N^{2/3}\right)\nonumber\\&+ \Pb\left(Z^{{\rm stat},\lambda_+}(0)-\xi_\lambda N<c_3N^{2/3}\right).
\end{align}
In the same way as above, we obtain that $Z^{{\rm stat},\lambda_{+}}(0)$ will concentrate around $\xi^{+}_{\lambda}N=\xi_{\lambda}N-(a^{\prime}_1 r N^{2/3}+a^{\prime}_2 r^2 N^{1/3})$ for some $a_1^\prime(\lambda)<0$ and $a_2^\prime(\lambda)<0$. Choose $c_3=|a^{\prime}_1|r/2$. Then by \eqref{eq:ExitLowerTail} of Corollary~\ref{ExitPCor}, we get
\begin{equation}
\Pb\left(Z^{{\rm stat},\lambda_+}(CN^{1/3})- \xi^{+}_{\lambda}N <c_3 N^{2/3}\right)\leq \tilde{K}_1e^{-\tilde{K}_1 r^2}
\end{equation}
uniformly for all large $N$, and by Lemma~\ref{ExitPointOfGeneralPercolation}, there is a $r_0$ such that
\begin{equation}
\Pb\left(Z_{\lambda}(CN^{1/3})-\xi_{\lambda} N>|a^{\prime}_1|r N^{2/3}\right)\leq \tilde{K}^{\prime}_1 e^{-\tilde{K}^{\prime}_1 r^2}
\end{equation}
for all $r>r_0$ uniformly for all large $N$. This completes the proof.
\end{proof}

For the next result, we need to define the rescaled process for the stationary case,
\begin{equation}\label{DefScaledLPPStat}
B^{\pm}_{N,\lambda}(u)= \frac{L^{{\rm stat},\lambda_\pm}_{{\cal L}^\infty_\lambda\to P(uN^{1/3})}-L^{{\rm stat},\lambda_\pm}_{{\cal L}^\infty_\lambda\to P(0)}+ \frac{1-2\lambda}{\lambda(1-\lambda)}uN^{1/3}}{N^{1/3}}.
\end{equation}
The first result focus on a comparison between $B^{\pm}_{N,\lambda}$ and $L^{{\rm resc},\lambda}_N$. The proof is geometric and uses the idea of Lemma~1 of~\cite{CP15b}.
\begin{prop}\label{SandwithLemma}
Recall the good event defined in (\ref{eq:GoodEvent}), namely
\begin{equation}
G_{N}(r):= \left\{Z^{{\rm stat},\lambda_-}(CN^{1/3})\leq Z_\lambda(0)\right\} \cap \{Z_\lambda(CN^{1/3})\leq Z^{{\rm stat},\lambda_+}(0)\}.
\end{equation}
Fix any two real numbers $0<u<v$ inside the interval $[0,C]$. Then, under the event $G_N(r)$, we have
\begin{align}\label{eq:NewEventProperty}
B^{-}_{N,\lambda}(u)-B^{-}_{N,\lambda}(v)\leq L^{{\rm resc},\lambda}_{N}(u)- L^{{\rm resc},\lambda}_{N}(v) \leq B^{+}_{N,\lambda}(u) - B^{+}_{N,\lambda}(v).
\end{align}
\end{prop}
\begin{proof}
To prove this result, we use Lemma~\ref{CouplingLemma}. We start by noting that for any two positive real numbers $a<b$, the following hold
\begin{align}\label{eq:ExitPointOderClaim}
Z_\lambda(a)\leq Z_\lambda(b)\quad \textrm{and} \quad Z^{{\rm stat},\lambda_\pm }(a)\leq Z^{{\rm stat},\lambda_\pm}(b)
\end{align}
with probability $1$.
The way we prove these inequalities is by contradiction. If any of the inequalities in \eqref{eq:ExitPointOderClaim} does not hold, then the corresponding last passage paths would intersect at some point (lets say $\mathbf{c}$) in between. In such case, the sections of the last passage paths intercepted between ${\cal L}_{\lambda}$ to $\mathbf{c}$ is actually the last passage path from ${\cal L}_{\lambda}$ to $\mathbf{c}$. If $a\neq b$, then this would imply the existence of two different last passage paths $\pi_{\mathcal{L}_{\lambda}\to \mathbf{c}}$ (or, $\pi^{\varrho}_{\mathcal{L}_{\lambda}\to \mathbf{c}}$) from $\mathcal{L}_{\lambda}$ to the point $\mathbf{c}$. As a consequence, the uniqueness of the last passage paths will be contradicted. Thus, the claim follows.

Now, notice that the condition $\mathbf{(i)}$ $Z_\lambda(CN^{1/3})\leq Z^{{\rm stat},\lambda_+}(0)$ implies
\begin{align}\label{eq:OneSidedImplication}
Z_\lambda(vN^{1/3})\leq Z_\lambda(CN^{1/3})\leq Z^{{\rm stat},\lambda_+}(0)\leq Z^{{\rm stat},\lambda_+}(uN^{1/3})
\end{align}
whereas the condition $\mathbf{(ii)}$ $Z^{{\rm stat},\lambda_-}(CN^{1/3})\leq Z_\lambda(0)$ implies
\begin{align}\label{eq:OtherSidedImplication}
Z^{{\rm stat},\lambda_-}(vN^{1/3})\leq Z^{{\rm stat},\lambda_-}(CN^{1/3})\leq Z_\lambda(0)\leq Z_\lambda(uN^{1/3}).
\end{align}
Employing Lemma~\ref{CouplingLemma}, we obtain the following inequalities
\begin{equation}
\begin{aligned}
 L^{{\rm resc},\lambda}_{N}(u)- L^{{\rm resc},\lambda}_{N}(v)& \leq B^{+}_{N,\lambda}(u) - B^{+}_{N,\lambda}(v),\\
 L^{{\rm resc},\lambda}_{N}(u)- L^{{\rm resc},\lambda}_{N}(v) &\geq B^{-}_{N,\lambda}(u)-B^{-}_{N,\lambda}(v)
\end{aligned}
\end{equation}
from \eqref{eq:OneSidedImplication} and \eqref{eq:OtherSidedImplication} respectively.
Hence, the proof is completed.
\end{proof}

\subsubsection{Proof of Proposition~\ref{Tightness}}
Now, we turn to proving the Proposition~\ref{Tightness}. For this, we closely follows the arguments for tightness in~\cite{Pi17}, Theorem~1.
First, we fix some large $C>0$. For any $\epsilon>0$, denote $A_\epsilon:= \{\max_{u_1,u_2\in[-C,C]} \left| L^{{\rm resc},\lambda}_{N}(u_1)- L^{{\rm resc},\lambda}_{N}(u_2)\right|\geq \epsilon\}$. It suffices to show that $\Pb(A_\epsilon )$ goes to $0$ as $N\to\infty$.
We have
\begin{align}\label{eq:SplittingProbability}
\Pb\left(A_\epsilon\right)\leq \Pb\left(A_\epsilon\cap G_N(r)\right) + \Pb((G_N(r))^c).
\end{align}
\emph{Bound on $\Pb\left(A_\epsilon\cap G_N(r)\right)$.}
 Whenever $G_N(r)$ holds, then by Proposition~\ref{SandwithLemma}, we have
 \begin{equation}\label{eq:Dominition}
 \begin{aligned}
& \max_{u_1,u_2\in[-C,C]} \left|L^{{\rm resc},\lambda}_{N}(u_1)- L^{{\rm resc},\lambda}_{N}(u_2)\right| \\
 \leq &\max\left\{\max_{u_1,u_2\in[-C,C]} \left|B^{+}_{N,\lambda}(u_1) - B^{+}_{N,\lambda}(u_2)\right|, \max_{u_1,u_2\in[-C,C]} \left|B^{-}_{N,\lambda}(u_1) - B^{-}_{N,\lambda}(u_2)\right|\right\}.
 \end{aligned}
 \end{equation}
The increments for the stationary case can be expressed, by Lemma~4.2 of~\cite{BCS06}, as
\begin{align}\label{eq:IIDIncrement}
L^{\lambda_\pm}_{{\cal L}_{\lambda}^\infty\to P(uN^{1/3})}- L^{\lambda_\pm}_{{\cal L}_{\lambda}^\infty\to P(0)}
\sim \left\{\begin{matrix}
\sum_{i=1}^{\lfloor uN^{1/3}\rfloor}p_i -\sum_{i=1}^{\lfloor uN^{1/3}\rfloor} q_i, & \text{for } u>0,\\
\sum_{i=\lfloor uN^{1/3}\rfloor}^{-1} q_{i} - \sum_{i=\lfloor uN^{1/3}\rfloor}^{-1}p_i, & \text{for } u<0,
\end{matrix}\right.
\end{align}
where $p_i$s (resp.\ $q_i$s) are i.i.d.\ $\mbox{Exp}(1-\lambda_{\pm})$ (resp.\ $\mbox{Exp}(\lambda_{\pm})$) random variables.

We have
\begin{equation}\label{eq5.92}
\E\left(L^{\lambda_\pm}_{{\cal L}_{\lambda}^\infty\to P(uN^{1/3})}- L^{\lambda_\pm}_{{\cal L}_{\lambda}^\infty\to P(0)}\right) =\frac{2\lambda_\pm-1}{\lambda_\pm(1-\lambda_\pm)} uN^{1/3}
\end{equation}
and $\frac{2\lambda_\pm-1}{\lambda_\pm(1-\lambda_\pm)} =\frac{2\lambda-1}{\lambda(1-\lambda)}+K(\lambda)rN^{-1/3}$ for $K(\lambda)=\frac{1-2\lambda+2\lambda^2}{(1-\lambda)^2\lambda^2}+\Or(N^{-1/3})$. Define
\begin{equation}\label{eq5.92b}
\mathcal{B}^{\pm}_{N,\lambda}(u)=\frac{L^{\lambda_\pm}_{{\cal L}_{\lambda}^\infty\to P(uN^{1/3})}- L^{\lambda_\pm}_{{\cal L}_{\lambda}^\infty\to P(0)}-\E\left(L^{\lambda_\pm}_{{\cal L}_{\lambda}^\infty\to P(uN^{1/3})}- L^{\lambda_\pm}_{{\cal L}_{\lambda}^\infty\to P(0)}\right)}{N^{1/6}}.
\end{equation}
Then by Donsker's theorem, we have
\begin{equation}\label{eq5.93}
\lim_{N\to\infty}\mathcal{B}^{\pm}_{N,\lambda}(u)\Rightarrow \mathcal{B}^{\pm}(u),
\end{equation}
where '$\Rightarrow$' denotes the convergence in distribution and $\mathcal{B}^{\pm}(\cdot)$ is a two Brownian motions (two sided). The slightly different centering in (\ref{eq5.92}) and (\ref{DefScaledLPPStat}) leads to
\begin{equation}\label{eq5.94}
N^{1/6} B^{\pm}_{N,\lambda}(u)=\mathcal{B}^\pm_{N,\lambda}(u)+K(\lambda)u r N^{-1/6}.
\end{equation}
This implies that as $N\to\infty$, $B^{\pm}_{N,\lambda}(u) \sim N^{-1/6} {\cal B}^\pm(u)$. Hence, $\max_{u_1,u_2\in[-C,C]} \left|B^{+}_{N,\lambda}(u_1) - B^{+}_{N,\lambda}(u_2)\right|$ and $\max_{u_1,u_2\in[-C,C]} \left|B^{-}_{N,\lambda}(u_1) - B^{-}_{N,\lambda}(u_2)\right|$ converge to $0$ in probability. This, together with \eqref{eq:Dominition} implies $\limsup_{N\to \infty}\Pb\left(A_\epsilon\cap G_N(r)\right)=0$ for any choice of $r$.

Using the bound on $\Pb((G_N(r))^c)$ of Proposition~\ref{CrucialEventIdentification}, we finally obtain $\lim_{r\to\infty}\limsup_{N\to\infty} \Pb\left(A_\epsilon\right)=0$.

\subsubsection{Proof of Proposition~\ref{ConcBoundOnMaximum}}
We show (\ref{eq:ExpectedResult1}) for $v=0$. The claim for general $v\in [-C,C]$ is simply obtained by replacing $N$ by $N+v N^{1/3}$ in the proof and noting that $(N+vN^{1/3})^{1/3}=N^{1/3}+\Or(N^{-1/3})$. For any $\e>0$, define the event \mbox{$\widetilde A_\e=\{\max_{u\in [-C,C]}|L^{{\rm resc},\lambda}_N(u)-L^{{\rm resc},\lambda}_N(0)|\geq \e\}$}. Then
\begin{equation}
\Pb(\widetilde A_\e)\leq \Pb((G_N(r))^c)+\Pb(\widetilde A_\e\cap G_N(r)).
\end{equation}
Applying Proposition~\ref{CrucialEventIdentification} with $r=N^{1/6-\delta/2}/\sqrt{C}$ leads to $\Pb((G_N(r))^c)\leq K^{\prime}_1e^{-K^{\prime}_2 N^{1/3-\delta}/C}$. Thus it remains to get a bound of $\Pb(\widetilde A_\e\cap G_N(r))$ for such choice of $r$.

By Proposition~\ref{SandwithLemma}, when the event $G_N(r)$ holds, then the increment of the (scaled) LPP are bounded by the ones of the stationary case with modified. This implies
\begin{equation}\label{eq5.96}
\Pb(\widetilde A_\e\cap G_N(r)) \leq \Pb\left(\max_{u\in [-C,C]} B^{+}_{N,\lambda}(u)\geq \e\right) + \Pb\left(\min_{u\in [-C,C]} B^{-}_{N,\lambda}(u)\leq -\e\right).
\end{equation}
Recall the relation (\ref{eq5.94}), namely $N^{1/6} B^{\pm}_{N,\lambda}(u)=\mathcal{B}^\pm_{N,\lambda}(u)+K(\lambda)u r N^{-1/6}$. By the above choice of $r$, the term $K(\lambda)u r N^{-1/6}\sim N^{-\delta/2}\to 0$ as $N\to\infty$. Therefore to bound (\ref{eq5.96}) it is enough to bound
\begin{equation}
\Pb\left(\max_{u\in[0,C]}\mathcal{B}^+_{N,\lambda}(u)\geq \e N^{1/6}\right),\quad\Pb\left(\max_{u\in[0,C]}-\mathcal{B}^-_{N,\lambda}(u)\geq \e N^{1/6}\right)
\end{equation}
together with the same bound for $u\in [-C,0]$. $\mathcal{B}^+_{N,\lambda}$ is a rescaled version of a sum of i.i.d.\ random centered variables with exponentially decaying tails, thus it is a (two-sided) martingale. Each of the bound is obtained in the same way, thus we write the details only for one of each.
By Doob's maximal inequality for the submartingale $u\mapsto e^{\chi \mathcal{B}^{+}_{N,\lambda}(u)}$, and minimizing over $\chi$, we get
\begin{equation}\label{eq:MaximalInequality}
\Pb(\max_{u\in[0,C]}\mathcal{B}^+_{N,\lambda}(u)\geq \e N^{1/6})\leq \inf_{\chi>0} \exp\left\{-\chi sN^{1/6}+\ln \E[\exp(\chi\mathcal{B}^{+}_{N,\lambda}(C))]\right\}.
\end{equation}
A simple computation gives
\begin{equation}
\ln \E[\exp(\chi\mathcal{B}^{+}_{N,\lambda}(C))] = 2 \chi^2 C N^{1/3} (\lambda(1-\lambda))^{-1}+\Or(1).
\end{equation}
Plugging this into \eqref{eq:MaximalInequality} and minimizing with respect to $\chi\geq 0$, we get
\begin{equation}
\textrm{l.h.s.~of}~(\ref{eq:MaximalInequality})\leq K_1 \exp(-K_2 \e^2 N^{1/3}/C),
\end{equation}
where $K_1,K_2$ are two constants which depend on $\lambda$. This bound and the one above the probability on $G_N(r)^c$ implies (\ref{eq:ExpectedResult1}).

Now we prove the inequality \eqref{eq:ExpectedResult2} for $u=0$. The inequality for general $u\in[-C,C]$ is obtained by substituting $\gamma$ with $\gamma+uN^{-2/3}$.
This time we are dealing with LPP between points which are``time-like'', i.e., where both coordinates of the end-point increase. However, since we are looking at a distance from the reference point $\bar P(0,0)$ which is the of the same order as the case of \eqref{eq:ExpectedResult1}, we can get a similar bound by using the analysis used to get \eqref{eq:ExpectedResult1} twice. Let us decompose the increments in \eqref{eq:ExpectedResult2} into a sum two increment the LPP between ``space-like'' points, i.e., with one coordinate increasing and the other decreasing. We have
\begin{equation}\label{eq5.101}
\begin{aligned}
L^{{\rm resc},\lambda}_N(0,v)-L^{{\rm resc},\lambda}_N(0,0) &= \left(L^{{\rm resc},\lambda}_N(0,v)-L^{{\rm resc},\lambda}_N((1-\gamma)v,v)\right) \\
&+ \left(L^{{\rm resc},\lambda}_N((1-\gamma)v,v)-L^{{\rm resc},\lambda}_N(0,0)\right).
\end{aligned}
\end{equation}
The first term in the r.h.s.~of~\eqref{eq5.101} is using \eqref{eq:ExpectedResult1}. For the second term in the r.h.s.~of~\eqref{eq5.101} we get the same bound by applying exactly the same strategy for getting \eqref{eq:ExpectedResult1}. Indeed, all the ingredients having two end-points used in the proof of \eqref{eq:ExpectedResult1} extends to points which are space-like: see Lemma~4.2 of~\cite{BCS06} and the geometric proof of Proposition~\ref{SandwithLemma}. The detailed proof do not require any new ideas, so we skip the details here.


\end{document}